\documentclass[11pt]{article}

\usepackage{amssymb,amsfonts,amsmath}
\usepackage{enumerate}

\usepackage[english]{babel}
\usepackage[cp1250]{inputenc}

\usepackage{algorithm2e}

\usepackage{longtable}

\usepackage{graphicx}
\usepackage{index}
\usepackage{fancyhdr}
\usepackage{lhelp}
\usepackage{optparams}
\usepackage{psfrag}
\usepackage{forloop}
\usepackage{setspace}
\usepackage{tikz}
\usepackage{subcaption}
\usepackage{pgffor}
\usepackage{xifthen}

\usetikzlibrary{decorations.pathreplacing,automata,calc,positioning}

\newcommand{\pch}{\chi_{\rho}}

\newcommand{\qed}{\hfill $\square$ \bigskip}

\newcommand{\mptt}[1]{}

\newtheorem{theorem}{Theorem} 

\newtheorem{lemma}[theorem]{Lemma}

\newtheorem{proposition}[theorem]{Proposition}

\begin{document}

\title{\bf A characterization of $4$-$\pch$-(vertex-)critical graphs}

\author{Jasmina Ferme \footnote{E-mail: jasmina.ferme1@um.si. Adress: Faculty of Education, University of Maribor, Koro\v ska cesta 160, 2000 Maribor, Slovenia}}
\date{}
\maketitle
\begin{center}
Faculty of Education, University of Maribor, Slovenia\\
\medskip
Faculty of Natural Sciences and Mathematics, University of Maribor, Slovenia\\
\medskip
\end{center}

\begin{abstract}
Given a graph $G$, a function $c:V(G)\longrightarrow \{1,\ldots,k\}$ with the property that $c(u)=c(v)=i$ implies that the distance between $u$ and $v$ is greater than $i$, is called a $k$-packing coloring of $G$. The smallest integer $k$ for which there exists a $k$-packing coloring of $G$ is called the packing chromatic number of $G$, and is denoted by $\pch(G)$. Packing chromatic vertex-critical graphs are the graphs $G$ for which $\pch(G-x) < \pch(G)$ holds for every vertex $x$ of $G$. A graph $G$ is called a packing chromatic critical graph if for every proper subgraph $H$ of $G$, $\pch(H) < \pch(G)$. 
Both of the mentioned variations of critical graphs with respect to the packing chromatic number have already been studied \cite{ bf-2019, kr-2019}. All packing chromatic (vertex-)critical graphs $G$ with $\pch(G)=3$ were characterized, while there were known only 
partial results for graphs $G$ with $\pch(G)=4$. 
%
In this paper, we provide characterizations of all packing chromatic vertex-critical graphs $G$ with $\chi_\rho(G)=4$ and all packing chromatic critical graphs $G$ with $\chi_\rho(G)=4$. 
\end{abstract}


\medskip\noindent
{\bf AMS Subj.\ Class:} 05C15, 05C70, 05C12.

\section{Introduction}
There are many variations of graph coloring and one of them is packing coloring. Given a graph $G$ and a positive integer $i$, a \textit{$k$-packing coloring} is a mapping $c: V(G)\longrightarrow \{1,2, \ldots, k\}$ with the following property: if $c(u)=c(v)=i$, then $d(u, v) > i$ for any $u,v \in V(G)$ and $i \in \{1,2, \ldots, k\}$ (note that $d(u,v)$ is the usual shortest-path distance between $u$ and $v$). The {\em packing chromatic number} of $G$, denoted by $\pch(G)$, is the smallest integer $k$ such that there exists a $k$-packing coloring of $G$. 

The packing chromatic number was introduced by Goddard, S.~M.~Hedetniemi, S.~T.~Hedetniemi, Harris and Rall~\cite{goddard-2008} in 2008 under the name broadcast
chromatic number. The current name was given in the second paper on the topic by Bre\v sar, Klav\v zar and Rall~\cite{bkr-2007}. 
The concept has a very wide spectrum of potential applications, such as frequency assignments~\cite{goddard-2008} or applications
in resource placements and biological diversity~\cite{bkr-2007}. 

Packing coloring has attracted many authors. This is reflected in the (probably non-exhaustive) list of papers on this topic that were published only in the last two years~\cite{balogh-2019, bozovic, bfk, bgt-2020, deng, ffgm, ght-2019, hjk-2020, kr-2019, alf, vesel_korze} (see also a survey~\cite{survey}). 
One of the main areas of investigation has been to determine the boundedness or the exact values of the packing chromatic numbers of several classes of (finite and infinite) graphs~\cite{bf-2018a, bf-2018b, bkr-2007, bkr-2016, deng, ekstein-2014, ekstein-2010, fiala-2009, finbow-2010, korze-2014, lb-2017, alf, martin-2017, togni-2014, torres-2015}. 
Among them, a lot of attention was given to the question of boundedness of the invariant in the class of (sub)cubic graphs, which was already posed in the seminal paper. Finally, it was answered in the negative by Balogh, Kostochka and Liu~\cite{balogh-2018} (see also an explicit construction in~\cite{bf-2018b}). 
Also, the problem of boundedness or the exact values of the packing chromatic numbers for the infinite grids was studied by many authors~\cite{bkr-2007, ekstein-2010, fiala-2009, finbow-2010, korze-2014, martin-2017, soukal-2010}. For instance, the question of what is the packing chromatic
number of the infinite square grid was already considered  in the seminal paper~\cite{goddard-2008}. Later, a lower and an upper bound for this number were improved~\cite{ekstein-2010, fiala-2009, soukal-2010} and currently, it is known that the packing chromatic number of the square grid is between $13$ and $15$~\cite{martin-2017}.

It is clear that the packing chromatic number is hereditary: a graph cannot have smaller packing chromatic number than its subgraphs. In particular, if we delete a vertex $v$ from a given graph $G$, then for the obtained graph we have: $\pch(G-v) \leq \pch(G)$. 
Klav\v zar and Rall~\cite{kr-2019} investigated the class of graphs $G$ with the property that $\pch(G-v)<\pch(G)$ holds for every $v\in V(G)$. Such graphs are called {\em packing chromatic vertex-critical graph}, or shorter {\em$\chi_\rho$-vertex-critical graph}. In the case when $G$ is $\chi_\rho$-vertex-critical and $\chi_\rho(G)=k$, we also say that $G$ is {\em $k$-$\chi_\rho$-vertex-critical}. Among other results, the mentioned authors characterized $3$-$\pch$-vertex-critical graphs, provided a partial characterization of $4$-$\pch$-vertex-critical graphs and considered $\pch$-vertex-critical trees. 
Later, Bre\v sar and Ferme~\cite{bf-2019} studied a different (basic) version of critical graphs for the packing chromatic number. Namely, they considered the class of graphs $G$ satisfying the following property: $\pch(H)<\pch(G)$ for each proper subgraph $H$ of a graph $G$. Such graphs are called {\em packing chromatic critical graph}, or shorter {\em$\chi_\rho$-critical graph}. If $G$ is $\chi_\rho$-critical and $\chi_\rho(G)=k$, we can say that $G$ is {\em $k$-$\chi_\rho$-critical}. The mentioned authors characterized $\pch$-critical graphs with diameter 2, $\pch$-critical block graphs with diameter 3 and $3$-$\pch$-critical graphs.  They also considered $\pch$-critical trees. 
In both of the mentioned papers, a partial characterization of $4$-$\pch$-(vertex)-critical graphs is given. In this paper, we present a general characterization of such graphs. 

The paper is organized as follows. In the next section, we establish the notation and define the concepts used throughout the paper. We present the known family of graphs $G$ with $\pch(G)=3$, which will help us to characterize $4$-$\pch$-vertex-critical graphs. In addition, we prove some lemmas, which will be very useful for the proofs in the sequel of this paper. 
In Section \ref{sec:4kriticni}, we recall some partial characterizations of $4$-$\pch$-vertex-critical graphs. Then, we present all $4$-$\pch$-vertex-critical graphs and prove the characterization. 
Based on this result, in Section \ref{kriticni_zadnje}, we provide a complete characterization of $4$-$\pch$-critical graphs.
We end the paper with some remarks.

\section{Notation and preliminaries}


Let $G$ be a graph (unless stated otherwise, the term \textit{graph} refers to a simple graph). We denote its vertex set by $V(G)$ and its set of edges by $E(G)$. The \textit{(open) neighborhood} of an arbitrary vertex $v \in V(G)$, $N_G(v)$, is the set of all vertices adjacent to $v$. The \textit{closed neighborhood} of $v$ is $N_G(v) \cup \{v\}$ and is denoted by $N_G[v]$.
The number of elements in $N_G(v)$, $|N_G(v)|$, is called the \textit{degree} of $v$ and is denoted by $\deg_G(v)$. In the case when $\deg_G(v) = 1$, we say that $v$ is a \textit{leaf} or a \textit{pendant vertex}. If $\deg(v)>1$, then $v$ is called a \textit{non-pendant vertex}. An \textit{isolated vertex} is a vertex $v$ with $\deg_G(v)=0$.
Further, the \textit{distance} between two vertices $u,v \in V(G)$, denoted by  $d_G(u,v)$, is the length of a shortest $u$-$v$-path in $G$. 
Note that the subscript in some of the above notations may be omitted if the graph $G$ is clear from the context.

A graph $H$ is a \textit{subgraph} of $G$ if $V(H) \subseteq V(G)$ and $E(H) \subseteq E(G)$. A subgraph $H$ of a graph $G$ is called a \textit{proper subgraph} of $G$ if $V(H)$ is a proper subset of $V(G)$ or $E(H)$ is a proper subset of $E(G)$. If $H$ is a subgraph of $G$ and $V(G)=V(H)$, then we say that $H$ is a \textit{spanning subgraph} of $G$. 
Further, a subgraph $H$ of a graph $G$ is an \textit{induced subgraph} of $G$ if for each pair of the vertices $a,b \in V(H)$ holds the following: if $ab \in E(G)$, then $ab \in E(H)$. If $A \subset V(G)$, then $G[A]$ denotes an induced subgraph of $G$ with the vertex set $A$. We will also use the notation $G-v$, which is shorter for $G[V(G)\setminus \{v\}]$. Next, if $e=xy \in E(G)$, then $G-e$ (or $G-xy$) denotes the subgraph of $G$ with $V(G-e)=V(G)$ and $E(G-e)=E(G) \setminus \{e\}$.

Let $i$ be a positive integer. Recall that a \textit{$k$-packing coloring} of $G$ is a mapping $c: V(G)\longrightarrow \{1,2, \ldots, k\}$ with the following property: if $c(u)=c(v)=i$, then $d(u, v) > i$ for every $u,v \in V(G)$ and $i \in \{1,2, \ldots, k\}$. We say that $G$ is \textit{$k$-packing colorable} if there exists a $k$-packing coloring of $G$. The {\em packing chromatic number} of $G$, denoted by $\pch(G)$, is the smallest integer $k$ such that there exists a $k$-packing coloring of $G$. 
We have already mentioned that the packing chromatic number is hereditary, which means that for any subgraph $H$ of $G$, $\chi_\rho(H) \leq \chi_\rho(G)$. This property will be used several times in the sequel of this paper.
Recall that $\pch(K_n)=n$ for every complete graph $K_n$, $n \geq 1$. Further, let $C_n$ be a cycle of order $n \geq 3$. Then, $\pch(C_n)=3$ if $n=3$ or $n$ is divisible by $4$. Otherwise, $\pch(C_n)=4$. For any path $P_n$ we have: $\pch(P_1)=1$, $\pch(P_2)= \pch(P_3)=2$ and $\pch(P_n)=3$ if $n \geq 4$.  \\

Now, we prove two lemmas, which will be very useful in Sections \ref{sec:4kriticni} and \ref{kriticni_zadnje}.\\

Let $n$ be a positive integer. The graph $X_n$ is formed from the disjoint union of one copy of $K_3$ and one copy of $P_n$ by joining a vertex of $K_3$ and a leaf (or an isolated vertex) of $P_n$ (see Fig. \ref{x5}). Similarly, $Y_n$ is the graph obtained from the disjoint union of one copy of $C_4$ and one copy of $P_n$ by joining a vertex of $C_4$ and a leaf (or an isolated vertex) of $P_n$ (see Fig. \ref{y5}). 

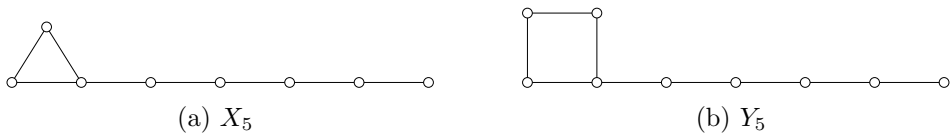
\begin{figure}[h]
\centering	
\begin{subfigure}[b]{0.4\textwidth}
       \centering
        \resizebox{\linewidth}{!}{
\begin{tikzpicture}
\def\vr{2pt}
\def\len{1}

\foreach \i in {1, 2, 3, 4, 5, 6, 7}{
\coordinate(x_\i) at (\i, 0); }
\coordinate(a) at (1.5, 0.8);
\draw (x_1) -- (x_2) -- (x_3)--(x_4)--(x_5)--(x_6)--(x_7);
\draw (x_1)--(a)--(x_2);
\foreach \i in {1, 2, 3, 4, 5, 6, 7}{
\draw (x_\i)[fill=white]circle(\vr);}
\draw (a)[fill=white]circle(\vr);
\end{tikzpicture}}
\caption{$X_5$}
\label{x5}
\end{subfigure}
\hspace{0.8cm}
\begin{subfigure}[b]{0.4\textwidth}
       \centering
        \resizebox{\linewidth}{!}{
\begin{tikzpicture}
\def\vr{2pt}
\def\len{1}

\foreach \i in {1, 2, 3, 4, 5, 6, 7}{
\coordinate(x_\i) at (\i, 0); }
\coordinate(a) at (1, 1);
\coordinate(b) at (2, 1);
\draw (x_1) -- (x_2) -- (x_3)--(x_4)--(x_5)--(x_6)--(x_7);
\draw(x_1)--(a)--(b)--(x_2);

\foreach \i in {1, 2, 3, 4, 5, 6, 7}{
\draw (x_\i)[fill=white]circle(\vr);}
\draw (a)[fill=white]circle(\vr);
\draw (b)[fill=white]circle(\vr);
\end{tikzpicture}}
\caption{$Y_5$}
\label{y5}
\end{subfigure}
\caption{Graphs $X_5$ and $Y_5$}
\label{fig:xn_yn}
\end{figure}

\begin{lemma}
If $n \geq 1$, then $ \pch(X_n)=3$ and $\pch(Y_n)=3$.
\label{lema:trikotnik+pot=3barve}
\label{lema:c4+pot=3barve}
\end{lemma}

\begin{proof}
Let $n \geq 1$ be an arbitrary positive integer. Recall that $\pch(K_3)=3$ and $\pch(C_4)=3$. Since $X_n$ contains a subgraph isomorphic to $K_3$ and $Y_n$ contains a subgraph isomorphic to $C_4$, the hereditary property of the packing chromatic number implies that $\pch(X_n) \geq 3$ and $\pch(Y_n) \geq 3$. 

On the other hand, there exist $3$-packing colorings of $X_n$ and $Y_n$, defined as follows. Color the vertices belonging to $K_3$ in $X_n$ with different colors from $\{1,2,3\}$ such that the vertex of degree $3$ receives color $3$. Similarly, color the vertices belonging to $C_4$ in $Y_n$ with different colors from $\{1,2,3\}$ such that the vertex of degree $3$ receives color $3$. 
Next, in both cases, color the vertices belonging to the path one after another (starting with the vertex, which is adjacent to the vertex of degree $3$) using the following pattern of colors: $1,2,1,3$. Since the described colorings are $3$-packing colorings of $X_n$, respectively $Y_n$, we derive that $\pch(X_n)=3$ and $\pch(Y_n)=3$. 
\qed 
\end{proof}


\begin{lemma}
If $T$ is the graph in Fig.~\ref{fig:T}, then $\pch(T)=3$.
\label{drevoT=3_barve}
\end{lemma}

\begin{proof}
Let $T$ be the graph from Fig.~\ref{fig:T}. 
Since $T$ has at least one edge, is a connected graph and is not a star, $\pch(T) \geq 3$~\cite{goddard-2008}. Clearly, there exists a $3$-packing coloring $c'$ of $T$ defined as follows: $c'(y')=3$, $c'(a)=c'(b)=c'(c)=1$ and $c'(d)=c'(e)=2$ (see Fig.~\ref{fig:T} for the notation of  the vertices). Thus, $\pch(T)=3$ (see also~\cite{goddard-2008}). 
\qed
\end{proof}

%
%
\begin{figure}[htb!]
\begin{center}
\begin{tikzpicture}
\def\vr{3pt}
\def\len{1}
	
\coordinate(x_1) at (0, 0); 
\coordinate(x_2) at (1, 0); 
\coordinate(x_3) at (2, 0); 
\coordinate(x_4) at (3, 0); 
\coordinate(x_5) at (4, 0); 
\coordinate(x_6) at (2, 1); 

\draw (x_1) -- (x_2) -- (x_3)--(x_4)--(x_5);
\draw (x_6)--(x_3);

\foreach \i in {1, 2, 3, 4, 5, 6}{
\draw (x_\i)[fill=white]circle(\vr);}

\draw(x_1)node[below]{$d$};
\draw(x_2)node[below]{$b$};
\draw(x_3)node[below]{$y'$};
\draw(x_4)node[below]{$c$};
\draw(x_5)node[below]{$e$};
\draw(x_6)node[left]{$a$};

\end{tikzpicture}
\end{center}
\caption{Graph $T$}
\label{fig:T}
\end{figure}


Recall the characterization of graphs $G$ with $\chi_\rho(G)=3$, which was proven by Goddard and co-authors~\cite{goddard-2008}.
Note that \textit{$T$-add to a vertex $v$} is formed as follows. First, we introduce a vertex $w_v$ and a set $X_v$ of independent vertices. Then, we add the edge $vw_v$ and some of the edges between $\{v,w_v\}$ and $X_v$.  

\begin{proposition}~\cite{goddard-2008}
Let $G$ be a graph. Then, $\chi_\rho(G)=3$ if and only if $G$ can be formed by taking some bipartite multigraph $H$ with bipartition $(U_1,U_3)$, subdividing every edge exactly once, adding leaves to some vertices in $U_1\cup U_3$, and then performing a single $T$-add to some vertices in $U_3$.
\label{goddard-pch=3}
\end{proposition}

In this paper, we denote the family of graphs with packing chromatic number $3$ by $\mathcal{G}_3$. \\

Let $G \in \mathcal{G}_3$ be an arbitrary graph. We denote the subsets of $V(G)$ as follows. If $G$ is obtained from a bipartite multigraph $H$ with bipartition $(U_1, U_3)$ by subdividing every edge of $H$ exactly once, adding leaves to some vertices in $U_1\cup U_3$, and performing a single $T$-add to some vertices in $U_3$, then:
\begin{itemize}
\item $V_1=U_1$;
\item $V_3=U_3$;
\item $V_2$ is the set of all vertices obtained by subdivision;
\item $V_0$ is the set of all leaves added to the vertices in $V_1$;
\item $V_4$ is the set of all leaves adjacent to the vertices from $V_3$; 
\item $V_5$ is the set of all vertices from $T$-adds, which belong to a subgraph of $G$ isomorphic to $K_3$ and have degree $2$ (therefore, $V_5$ is the set of all vertices, which belong to a subgraph of $G$ isomorphic to $K_3$ and have degree $2$);
\item $V_6$ is the set of all vertices from $T$-adds, which either: a) have degree at least $3$, or b) have degree $2$ and do not belong to a subgraph of $G$ isomorphic to $K_3$ (therefore, $V_6$ is the set of all vertices adjacent to the vertices from $V_3$, which have degree at least $3$, or have degree $2$, but do not belong to a subgraph of $G$ isomorphic to $K_3$ and are not obtained by subdivision);
\item $V_7$ is the set of all vertices from $T$-adds, which are leaves and are adjacent to the vertices from $V_6$ (therefore, $V_7$ consists of all leaves adjacent to the vertices from $V_6$).
\end{itemize}

An example of the described labeling of subsets of $V(G)$, $G \in \mathcal{G}_3$, is shown in Fig.~\ref{fig:mnozice_pch=3}. 

We observe that $V(G)$ is partitioned into sets $V_0, V_1, \ldots, V_7$.
Further, since $V_1 \cup V_3$ induces a subgraph of $G$  isomorphic to a bipartite multigraph in which each edge is subdivided exactly once, $d(v_1, v_1')=4k_1$, $k_1 \geq 1$, and $d(v_3, v_3')=4k_3$, $k_3 \geq 1$, for any $v_1,v_1' \in V_1$ and $v_3, v_3' \in V_3$. Next, suppose that $a \in V(G)$ belongs to a subgraph of $G$ isomorphic to $K_3$. This implies that $a \in V_3 \cup V_5 \cup V_6$. More precisely, if deg$(a)=2$, then $a \in V_5$. If deg$(a) \geq 3$, and each vertex from $N(a)$ is either a leaf or belongs to a triangle, then $a \in V_6$. Otherwise, $a \in V_3$. 
In addition, we observe that each vertex from $V_6$ has at most one neighbour which is not a leaf and does not belong to a triangle.\\

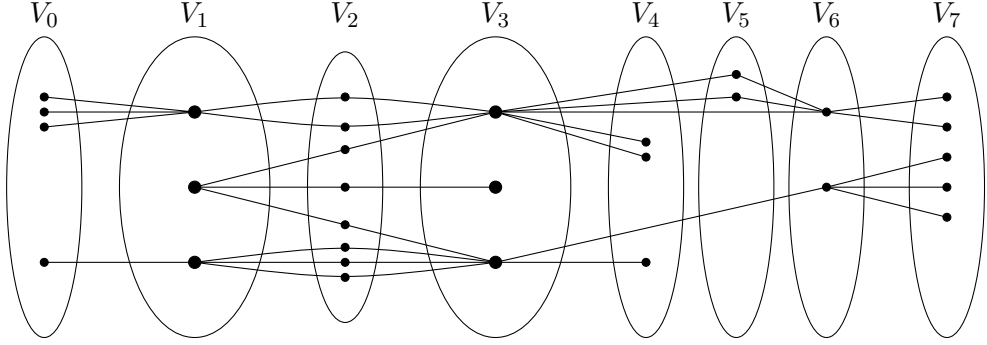
\begin{figure}[h]
\begin{center}
\begin{tikzpicture}
\def\vr{2.3pt}
\def\len{1}

\draw (0,0) ellipse (0.5cm and 2cm);
\coordinate(V_0) at (0, 2);
\draw(V_0)node[above]{$V_0$}; 

\draw (2,0) ellipse (1cm and 2cm);
\coordinate(V_1) at (2, 2);
\draw(V_1)node[above]{$V_1$}; 

\draw (4,0) ellipse (0.5cm and 1.8cm);
\coordinate(V_2) at (4, 2);
\draw(V_2)node[above]{$V_2$}; 

\draw (6,0) ellipse (1cm and 2cm);
\coordinate(V_3) at (6, 2);
\draw(V_3)node[above]{$V_3$}; 

\draw (8,0) ellipse (0.5cm and 2cm);
\coordinate(V_4) at (8, 2);
\draw(V_4)node[above]{$V_4$}; 

\draw (9.2,0) ellipse (0.5cm and 2cm);
\coordinate(V_5) at (9.2, 2);
\draw(V_5)node[above]{$V_5$}; 

\draw (10.4,0) ellipse (0.5cm and 2cm);
\coordinate(V_6) at (10.4, 2);
\draw(V_6)node[above]{$V_6$}; 

\draw (12,0) ellipse (0.5cm and 2cm);
\coordinate(V_7) at (12, 2);
\draw(V_7)node[above]{$V_7$}; 

\foreach \i in {1,2,3}{
\coordinate(v_\i) at (2,2-\i);
\draw(v_\i)[fill=black] circle(\vr);
\coordinate(u_\i) at (6,2-\i);
\draw(u_\i)[fill=black] circle(\vr);
}

\coordinate(a_1) at (0,1.2);
\draw(a_1)[fill=black] circle(1.5pt);
\coordinate(a_2) at (0,1);
\draw(a_2)[fill=black] circle(1.5pt);
\coordinate(a_3) at (0,0.8);
\draw(a_3)[fill=black] circle(1.5pt);
\coordinate(a_4) at (0,-1);
\draw(a_4)[fill=black] circle(1.5pt);
\draw[black](a_1)--(v_1)--(a_2); 
\draw[black] (a_3)--(v_1);
\draw[black] (a_4)--(v_3);

\coordinate(b_1) at (4,1.2);
\draw(b_1)[fill=black] circle(1.5pt);
\coordinate(b_2) at (4,0.8);
\draw(b_2)[fill=black] circle(1.5pt);
\draw (v_1) .. controls (4,1.25) .. (u_1);
\draw (v_1) .. controls (4, 0.75) .. (u_1);
\coordinate(b_3) at (4,0);
\draw(b_3)[fill=black] circle(1.5pt);
\coordinate(b_2) at (4,0.8);
\draw (v_2) .. controls (b_3) .. (u_2);
\coordinate(b_5) at (4,-1.2);
\draw(b_5)[fill=black] circle(1.5pt);
\coordinate(b_4) at (4,-0.8);
\draw(b_4)[fill=black] circle(1.5pt);
\coordinate(b_6) at (4,-1);
\draw(b_6)[fill=black] circle(1.5pt);
\draw (v_3) .. controls (4,-1.25) .. (u_3);
\draw (v_3) .. controls (4, -0.75) .. (u_3);
\draw (v_3) .. controls (4, -1) .. (u_3);
\coordinate(b_7) at (4,0.5);
\draw(b_7)[fill=black] circle(1.5pt);
\draw (v_2) -- (b_7) -- (u_1);
\coordinate(b_8) at (4,-0.5);
\draw(b_8)[fill=black] circle(1.5pt);
\draw (v_2) -- (b_8) -- (u_3);

\coordinate(c_1) at (9.2,1.5);
\draw(c_1)[fill=black] circle(1.5pt);
\coordinate(c_2) at (9.2,1.2);
\draw(c_2)[fill=black] circle(1.5pt);
\coordinate(d_1) at (10.4,1);
\draw(d_1)[fill=black] circle(1.5pt);
\coordinate(e_1) at (8,0.6);
\draw(e_1)[fill=black] circle(1.5pt);
\coordinate(e_2) at (8,0.4);
\draw(e_2)[fill=black] circle(1.5pt);

\draw (u_1)--(c_1)--(d_1)--(u_1);
\draw (u_1)--(c_2)--(d_1);
\draw (e_1)--(u_1)--(e_2);

\coordinate(e_3) at (8,-1);
\draw(e_3)[fill=black] circle(1.5pt);
\draw (u_3)--(e_3);

\coordinate(d_4) at (10.4,0);
\draw(d_4)[fill=black] circle(1.5pt);
\draw (u_3)--(d_4);

\coordinate(f_1) at (12,1.2);
\draw(f_1)[fill=black] circle(1.5pt);
\draw (d_1)--(f_1);
\coordinate(f_2) at (12,0.8);
\draw(f_2)[fill=black] circle(1.5pt);
\draw (d_1)--(f_2);

\coordinate(f_3) at (12,0.4);
\coordinate(f_4) at (12,0);
\coordinate(f_5) at (12,-0.4);
\draw(f_3)[fill=black] circle(1.5pt);
\draw(f_4)[fill=black] circle(1.5pt);
\draw(f_5)[fill=black] circle(1.5pt);
\draw (d_4)--(f_3);
\draw (d_4)--(f_4);
\draw (d_4)--(f_5);

\end{tikzpicture}
\end{center}
\caption{The sets $V_0, V_1, \ldots, V_7$} 
\label{fig:mnozice_pch=3}
\end{figure}

We end this section with two lemmas, which will be very useful in the sequel of this paper. 

\begin{lemma}
Let $G$ be the graph obtained by attaching a vertex to two adjacent vertices of $C_n, n \geq 4$.  Then, $G$ contains a subgraph isomorphic to $C_n$, $n \geq 5$, $n \not\equiv 0 \pmod{4}$.
\label{lemma:cikel+trikotnik=neustrezen_cikel}
\end{lemma}

\begin{proof}
Let $G$ be formed by attaching a vertex $v$ to two adjacent vertices of $C_n$, $n \geq 4$, with $V(C_n)=\{x_1,x_2, \ldots, x_n\}$ and $E(C_n)=\{x_1x_2, x_2x_3, \ldots, x_{n-1}x_n, x_nx_1\}$. Further, let $x_1v,x_2v \in E(G)$.
Clearly, if $n \not\equiv 0 \pmod{4}$, then we are done. Otherwise, $G$ contains a subgraph isomorphic to a cycle $C_{n+1}$ ($V(C_{n+1})=\{v,x_1,x_2, \ldots, x_n\}$ and $E(C_{n+1})=\{x_1v, vx_2, x_2x_3, \ldots, x_{n-1}x_n, x_nx_1\}$). Since $n+1 \geq 5$ and $n+1 \not\equiv0\pmod{4}$, our claim holds. 
\qed
\end{proof}

\begin{lemma}
Let $G$ be a graph, which does not contain a sugraph isomorphic to $C_n$, $n \geq 4$, $n \not \equiv 0\pmod{4}$. Next, let $u \in V(G)$ be an arbitrary vertex and let $N_i= \{v \in V(G);~d(u,v)=i\}$ for any $i \geq 1$. Further, let $i$ be an arbitrary positive integer and let $a,b$ be any two vertices from $N_i$. Then, $ab \notin E(G)$ if $ab$ is not an edge of a subgraph of $G$ isomorphic to $K_3$. 
\label{lemma:vodoravna_povezava}
\end{lemma}

\begin{proof}
Let $G$ be a graph, which does not contain a sugraph isomorphic to $C_n$, $n \geq 4$, $n \not \equiv 0\pmod{4}$, $u \in V(G)$ and $N_i= \{v \in V(G);~ d(u,v)=i\}$ for any $i \geq 1$. Next, let $i$ be an arbitrary positive integer and let $a,b$ be any two vertices from $N_i$. 
Suppose to the contrary that $ab \in E(G)$ and $ab$ is not an edge of a subgraph of $G$ isomorphic to $K_3$.   Denote by $P$ a shortest $a$-$u$-path and by $Q$ a shortest $b$-$u$-path . Let $w \in V(P) \cap V(Q)$ such that the distance between $w$ and $u$ is the largest possible. Then, the vertices from $P$ between $w$ and $a$, and the vertices from $Q$ between $w$ and $b$ form a cycle of odd order. Since $ab$ is not an edge of a triangle, we infer that $G$ contains a cycle of order $n > 3$, $n \not \equiv 0\pmod{4}$. This contradicts to our assumption.  
\qed
\end{proof}


\section{$4$-$\pch$-vertex-critical graphs}
\label{sec:4kriticni}

In this section, we study $4$-$\pch$-vertex-critical graphs. First, we recall some known partial results for these graphs. We follow with the main result of this paper: a complete characterization of $4$-$\pch$-vertex-critical graphs.



The first known partial characterization of $4$-$\pch$-vertex-critical graphs considers all graphs that contain a cycle $C_n$, where $n\ge 5$ is not divisible by $4$. 
Note that for such cycles, we have $\pch(C_n)=4$ and $\pch(C_n-u)=3$ for any $u \in V(C_n)$, which implies that these cycles themselves are $4$-$\pch$-vertex-critical. However, if $G$ has $\pch(G)=4$ and $G$ contains a non-spanning subgraph isomorphic to a cycle $C_n$, where $n\ge 5$ is not divisible by $4$, then $G$ is clearly not $4$-$\pch$-vertex-critical. 


Denote by $\mathcal{C}_5$ the family of graphs, which are shown in Fig.~\ref{fig:C5}, by $C_6$ the family of graphs, which are shown in Fig.~\ref{fig:C6}, and let $\mathcal{C}=\{C_n;~n \geq 5$, $n \not\equiv 0 \pmod{4}$.
The following theorem provides a complete characterization of $4$-$\pch$-vertex-critical graphs that contain a cycle $C_n \in \mathcal{C}$.

\begin{figure}[h]
\centering	
\begin{subfigure}[b]{0.15\textwidth}
       \centering
        \resizebox{\linewidth}{!}{
\begin{tikzpicture}
\def\vr{2pt}
\def\len{1}

\coordinate(a_1) at (0, 0); 
\coordinate(a_2) at (1.2, 0); 
\coordinate(a_3) at (1.7, 1.1); 
\coordinate(a_4) at (0.6, 2); 
\coordinate(a_5) at (-0.5, 1.1);
\draw (a_1) -- (a_2) -- (a_3)--(a_4)--(a_5)--(a_1);
\draw (a_2)--(a_4);
\foreach \i in {1, 2, 3, 4, 5}{
\draw (a_\i)[fill=white]circle(\vr);}
\end{tikzpicture}}
\end{subfigure}
\hspace{0.8cm}
\begin{subfigure}[b]{0.15\textwidth}
       \centering
        \resizebox{\linewidth}{!}{
\begin{tikzpicture}
\def\vr{2pt}
\def\len{1}

\coordinate(a_1) at (0, 0); 
\coordinate(a_2) at (1.2, 0); 
\coordinate(a_3) at (1.7, 1.1); 
\coordinate(a_4) at (0.6, 2); 
\coordinate(a_5) at (-0.5, 1.1);
\draw (a_1) -- (a_2) -- (a_3)--(a_4)--(a_5)--(a_1);
\draw (a_2)--(a_4);
\draw(a_1)--(a_3);
\foreach \i in {1, 2, 3, 4, 5}{
\draw (a_\i)[fill=white]circle(\vr);}
\end{tikzpicture}}
\end{subfigure}
\hspace{0.8cm}
\begin{subfigure}[b]{0.15\textwidth}
       \centering
        \resizebox{\linewidth}{!}{
\begin{tikzpicture}
\def\vr{2pt}
\def\len{1}

\coordinate(a_1) at (0, 0); 
\coordinate(a_2) at (1.2, 0); 
\coordinate(a_3) at (1.7, 1.1); 
\coordinate(a_4) at (0.6, 2); 
\coordinate(a_5) at (-0.5, 1.1);
\draw (a_1) -- (a_2) -- (a_3)--(a_4)--(a_5)--(a_1);
\draw (a_2)--(a_4)--(a_1);
\foreach \i in {1, 2, 3, 4, 5}{
\draw (a_\i)[fill=white]circle(\vr);}
\end{tikzpicture}}
\end{subfigure}
\hspace{0.8cm}
\begin{subfigure}[b]{0.15\textwidth}
       \centering
        \resizebox{\linewidth}{!}{
\begin{tikzpicture}
\def\vr{2pt}
\def\len{1}

\coordinate(a_1) at (0, 0); 
\coordinate(a_2) at (1.2, 0); 
\coordinate(a_3) at (1.7, 1.1); 
\coordinate(a_4) at (0.6, 2); 
\coordinate(a_5) at (-0.5, 1.1);
\draw (a_1) -- (a_2) -- (a_3)--(a_4)--(a_5)--(a_1);
\draw (a_2)--(a_4)--(a_1);
\draw(a_3)--(a_5);
\foreach \i in {1, 2, 3, 4, 5}{
\draw (a_\i)[fill=white]circle(\vr);}
\end{tikzpicture}}
\end{subfigure}

\caption{The graphs from $\mathcal{C}_5$}
\label{fig:C5}
\end{figure}
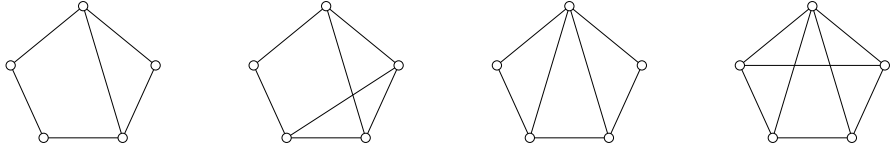

\begin{figure}[h]
\centering	
\begin{subfigure}[b]{0.15\textwidth}
       \centering
        \resizebox{\linewidth}{!}{
\begin{tikzpicture}
\def\vr{2pt}
\def\len{1}

\coordinate(a_1) at (0, 0); 
\coordinate(a_2) at (1.2, 0); 
\coordinate(a_3) at (1.7, 1); 
\coordinate(a_4) at (1.2, 2); 
\coordinate(a_5) at (0, 2); 
\coordinate(a_6) at (-0.5, 1);
\draw (a_1) -- (a_2) -- (a_3)--(a_4)--(a_5)--(a_6)--(a_1);
\draw (a_3)--(a_6);
\foreach \i in {1, 2, 3, 4, 5,6}{
\draw (a_\i)[fill=white]circle(\vr);}
\end{tikzpicture}}
\end{subfigure}
\hspace{0.8cm}
\begin{subfigure}[b]{0.15\textwidth}
       \centering
        \resizebox{\linewidth}{!}{
\begin{tikzpicture}
\def\vr{2pt}
\def\len{1}

\coordinate(a_1) at (0, 0); 
\coordinate(a_2) at (1.2, 0); 
\coordinate(a_3) at (1.7, 1); 
\coordinate(a_4) at (1.2, 2); 
\coordinate(a_5) at (0, 2); 
\coordinate(a_6) at (-0.5, 1);
\draw (a_1) -- (a_2) -- (a_3)--(a_4)--(a_5)--(a_6)--(a_1);
\draw (a_3)--(a_6);
\draw(a_2)--(a_5);
\foreach \i in {1, 2, 3, 4, 5,6}{
\draw (a_\i)[fill=white]circle(\vr);}
\end{tikzpicture}}
\end{subfigure}
\hspace{0.8cm}
\begin{subfigure}[b]{0.15\textwidth}
       \centering
        \resizebox{\linewidth}{!}{
\begin{tikzpicture}
\def\vr{2pt}
\def\len{1}

\coordinate(a_1) at (0, 0); 
\coordinate(a_2) at (1.2, 0); 
\coordinate(a_3) at (1.7, 1); 
\coordinate(a_4) at (1.2, 2); 
\coordinate(a_5) at (0, 2); 
\coordinate(a_6) at (-0.5, 1);
\draw (a_1) -- (a_2) -- (a_3)--(a_4)--(a_5)--(a_6)--(a_1);
\draw (a_3)--(a_6);
\draw(a_2)--(a_5);
\draw(a_1)--(a_4);
\foreach \i in {1, 2, 3, 4, 5,6}{
\draw (a_\i)[fill=white]circle(\vr);}
\end{tikzpicture}}
\end{subfigure}
\caption{The graphs from $\mathcal{C}_6$}
\label{fig:C6}
\end{figure}
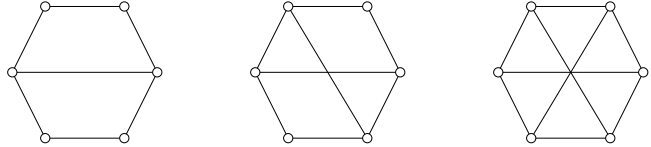

\begin{theorem}~\cite{kr-2019}
\label{thm:kriticni_CIKLI_rall_klavzar_izrek}
If $G$ is a graph that contains a cycle $C_n$, $n \geq 5$, $n \not\equiv 0\pmod{4}$, then $G$ is $4$-$\pch$-vertex-critical if and only if one of the following holds. 
\begin{itemize}
\item $G \in \mathcal{C}_5$;
\item $G \in \mathcal{C}_6$;
\item $G \in \mathcal{C}$.
\end{itemize}
\end{theorem}

We continue by presenting yet another partial characterization of $4$-$\pch$-vertex-critical graphs. 
Let $\mathcal{D}$ be the class of graphs that contain exactly one cycle and have an arbitrary number of leaves  attached to each of the vertices of the cycle. Next, recall that the {\em net graph} is obtained by attaching a single leaf to each vertex of $K_3$. 

\begin{theorem}~\cite{kr-2019}
\label{thm:kriticni_rall_klavzar_izrek}
A graph $G \in \mathcal{D}$ is a $4$-$\chi_\rho$-vertex-critical graph if and only if $G$ is one of the following graphs:
\begin{itemize}
\item $G \in \mathcal{C}$;
\item $G$ is the net graph;
\item $G$ is obtained by attaching a single leaf to two adjacent vertices of $C_4$;
\item $G$ is obtained by attaching a single leaf to two vertices at distance $3$ on $C_8$.
\end{itemize}
\end{theorem}


In the sequel of this section, we provide a complete characterization of $4$-$\pch$-vertex-critical graphs. Recall that each $\chi_\rho$-vertex-critical graph is connected~\cite{kr-2019}. \\

First, we prove that the graphs shown in Fig.~\ref{fig:4kriticni} are $4$-$\pch$-vertex-critical. 


\begin{figure}[h]
\begin{subfigure}[b]{0.15\textwidth}
       \centering
        \resizebox{\linewidth}{!}{
\begin{tikzpicture}
\def\vr{2pt}
\def\len{1}

\coordinate(a_1) at (0, 0); 
\coordinate(a_2) at (1, 0); 
\coordinate(a_3) at (1, 1); 
\coordinate(a_4) at (0, 1); 
\draw (a_1) -- (a_2) -- (a_3)--(a_4)--(a_1)--(a_3);
\draw (a_2)--(a_4);
\foreach \i in {1, 2, 3, 4}{
\draw (a_\i)[fill=white]circle(\vr);}
\draw(a_1)node[left]{\scriptsize{$a$}};
\draw(a_2)node[right]{\scriptsize{$b$}};
\draw(a_3)node[right]{\scriptsize{$c$}};
\draw(a_4)node[left]{\scriptsize{$d$}};
\end{tikzpicture}}
\caption{$K_4$}
\label{Trikotnik_prvi_K4}
\end{subfigure}
\hspace{0.25cm}
\begin{subfigure}[b]{0.23\textwidth}
       \centering
        \resizebox{\linewidth}{!}{
\begin{tikzpicture}
\def\vr{2pt}
\def\len{1}
\coordinate(b_1) at (0, 0); 
\coordinate(b_2) at (1, 0); 
\coordinate(b_3) at (2, 0); 
\coordinate(b_4) at (1, 1); 
\coordinate(b_5) at (0, 1);
\draw (b_1) -- (b_2) -- (b_3);
\draw (b_1)--(b_5)--(b_4)--(b_2);
\draw (b_1)--(b_4);
\foreach \i in {1, 2, 3, 4, 5}{
\draw (b_\i)[fill=white]circle(\vr);}
\draw(b_1)node[left]{\scriptsize{$a$}};
\draw(b_2)node[above right]{\scriptsize{$b$}};
\draw(b_3)node[right]{\scriptsize{$c$}};
\draw(b_4)node[right]{\scriptsize{$d$}};
\draw(b_5)node[left]{\scriptsize{$e$}};
\end{tikzpicture}}
\caption{$H_1$}
\label{Trikotnik_drugi}
\end{subfigure}
\hspace{0.25cm}
%
%
\begin{subfigure}[b]{0.20\textwidth}
       \centering
        \resizebox{\linewidth}{!}{
\begin{tikzpicture}
\def\vr{2pt}
\def\len{1}

\coordinate(c_1) at (0, 0); 
\coordinate(c_2) at (0, 1); 
\coordinate(c_3) at (0.8, 0.5); 
\coordinate(c_4) at (1.6, 0); 
\coordinate(c_5) at (1.6, 1);
\draw (c_1) -- (c_3) -- (c_4)--(c_5)--(c_3)--(c_2)--(c_1);
\foreach \i in {1, 2, 3, 4, 5}{
\draw (c_\i)[fill=white]circle(\vr);}
\draw(c_1)node[left]{\scriptsize{$a$}};
\draw(c_2)node[left]{\scriptsize{$e$}};
\draw(c_3)node[below]{\scriptsize{$b$}};
\draw(c_4)node[right]{\scriptsize{$c$}};
\draw(c_5)node[right]{\scriptsize{$d$}};
\end{tikzpicture}}
\caption{$H_2$}
\label{Trikotnik_tretji}
\end{subfigure}
\hspace{0.25cm}
\vspace{0.3cm}
%
%
\begin{subfigure}[b]{0.30\textwidth}
       \centering
        \resizebox{\linewidth}{!}{
\begin{tikzpicture}
\def\vr{2pt}
\def\len{1}

\coordinate(a_1) at (0, 0); 
\coordinate(a_2) at (1, 0); 
\coordinate(a_3) at (2, 0); 
\coordinate(a_4) at (3, 0); 
\coordinate(a_5) at (1.5, 0.8); 
\coordinate(a_6) at (1.5, 1.6); 
\draw (a_1) -- (a_2) -- (a_3)--(a_4);
\draw (a_2)--(a_5)--(a_3);
\draw (a_5)--(a_6);
\foreach \i in {1, 2, 3, 4, 5, 6}{
\draw (a_\i)[fill=white]circle(\vr);}
\draw(a_1)node[left]{\scriptsize{$d$}};
\draw(a_2)node[above left]{\scriptsize{$a$}};
\draw(a_3)node[above right]{\scriptsize{$b$}};
\draw(a_4)node[right]{\scriptsize{$e$}};
\draw(a_5)node[right]{\scriptsize{$c$}};
\draw(a_6)node[right]{\scriptsize{$f$}};
\end{tikzpicture}}
\caption{$H_3$}
\label{Trikotnik_cetrti}
\end{subfigure}
\hspace{0.25cm}
\vspace{0.3cm}
%
%
%
\begin{subfigure}[b]{0.47\textwidth}
       \centering
        \resizebox{\linewidth}{!}{
\begin{tikzpicture}
\def\vr{2pt}
\def\len{1}

\coordinate(a_1) at (0, 0); 
\coordinate(a_2) at (1, 0); 
\coordinate(a_3) at (2, 0); 
\coordinate(a_4) at (3, 0); 
\coordinate(a_5) at (4, 0); 
\coordinate(a_6) at (5, 0); 
\coordinate(a_7) at (2.5, 0.8); 
\draw (a_1) -- (a_2) -- (a_3)--(a_4)--(a_5)--(a_6);
\draw (a_3)--(a_7)--(a_4);
\foreach \i in {1, 2, 3, 4, 5, 6, 7}{
\draw (a_\i)[fill=white]circle(\vr);}
\draw(a_1)node[left]{\scriptsize{$a$}};
\draw(a_2)node[above]{\scriptsize{$b$}};
\draw(a_3)node[above left]{\scriptsize{$c$}};
\draw(a_4)node[above right]{\scriptsize{$d$}};
\draw(a_5)node[above]{\scriptsize{$e$}};
\draw(a_6)node[right]{\scriptsize{$f$}};
\draw(a_7)node[left]{\scriptsize{$g$}};
\end{tikzpicture}}
\caption{$H_4$}
\label{h4}
\end{subfigure}
\hspace{0.3cm}
\begin{subfigure}[b]{0.22\textwidth}
       \centering
        \resizebox{\linewidth}{!}{
\begin{tikzpicture}
\def\vr{2pt}
\def\len{1}

\coordinate(a_1) at (0, 0); 
\coordinate(a_2) at (1, 0); 
\coordinate(a_3) at (1, 1); 
\coordinate(a_4) at (0, 1); 
\coordinate(a_5) at (2, 0); 
\coordinate(a_6) at (2, 1); 
\draw (a_6)--(a_3)--(a_4)--(a_1) -- (a_2) -- (a_5);
\draw (a_2)--(a_3);
\foreach \i in {1, 2, 3, 4, 5, 6}{
\draw (a_\i)[fill=white]circle(\vr);}
\draw(a_1)node[left]{\footnotesize{$a$}};
\draw(a_2)node[above right]{\footnotesize{$b$}};
\draw(a_3)node[below right]{\footnotesize{$c$}};
\draw(a_4)node[left]{\footnotesize{$d$}};
\draw(a_5)node[right]{\footnotesize{$e$}};
\draw(a_6)node[right]{\footnotesize{$f$}};
\end{tikzpicture}}
\caption{$H_5$}
\label{h5}
\end{subfigure}
\hspace{0.2cm}
%
%
%
%
%
\begin{subfigure}[b]{0.21\textwidth}
       \centering
        \resizebox{\linewidth}{!}{
\begin{tikzpicture}
\def\vr{2pt}
\def\len{1}

\coordinate(a_1) at (0, 0); 
\coordinate(a_2) at (1, 0); 
\coordinate(a_3) at (1, 1); 
\coordinate(a_4) at (0, 1); 
\coordinate(a_5) at (2, 0); 
\coordinate(a_6) at (0.5, 0.5); 
\draw (a_2)--(a_3)--(a_4)--(a_1) -- (a_2) -- (a_5);
\draw (a_1)--(a_6)--(a_3);
\foreach \i in {1, 2, 3, 4, 5, 6}{
\draw (a_\i)[fill=white]circle(\vr);}
\draw(a_1)node[left]{\scriptsize{$a$}};
\draw(a_2)node[above right]{\scriptsize{$b$}};
\draw(a_3)node[right]{\scriptsize{$e$}};
\draw(a_4)node[left]{\scriptsize{$d$}};
\draw(a_5)node[above]{\scriptsize{$f$}};
\draw(a_6)node[left]{\scriptsize{$c$}};
\end{tikzpicture}}
\caption{$H_6$}
\label{h6}
\end{subfigure}
\hspace{0.8cm}
%
%
%
%
\begin{subfigure}[b]{0.43\textwidth}
       \centering
        \resizebox{\linewidth}{!}{
\begin{tikzpicture}
\def\vr{2pt}
\def\len{1}

\coordinate(a_1) at (0, 0); 
\coordinate(a_2) at (1, 0); 
\coordinate(a_3) at (2, 0); 
\coordinate(a_4) at (3, 0); 
\coordinate(a_5) at (4, 0); 
\coordinate(a_6) at (5, 0); 
\coordinate(a_7) at (4, 1); 
\coordinate(a_8) at (3, 1); 
\coordinate(a_9) at (2, 1); 
\coordinate(x) at (1, 1); 
\draw (a_1)--(a_2)--(a_3)--(a_4)--(a_5) -- (a_6);
\draw (a_5)--(a_7)--(a_8)--(a_9)--(x) -- (a_2);
\foreach \i in {1, 2, 3, 4, 5, 6, 7, 8, 9}{
\draw (a_\i)[fill=white]circle(\vr);}
\draw (x)[fill=white]circle(\vr);
\draw(a_1)node[above]{\scriptsize{$a$}};
\draw(a_2)node[above left]{\scriptsize{$b$}};
\draw(a_3)node[above]{\scriptsize{$c$}};
\draw(a_4)node[above]{\scriptsize{$d$}};
\draw(a_5)node[above right]{\scriptsize{$e$}};
\draw(a_6)node[above]{\scriptsize{$f$}};
\draw(a_7)node[right]{\scriptsize{$g$}};
\draw(a_8)node[below]{\scriptsize{$h$}};
\draw(a_9)node[below]{\scriptsize{$i$}};
\draw(x)node[left]{\scriptsize{$j$}};
\end{tikzpicture}}
\caption{$H_7$}
\label{h7}
\end{subfigure}
\vspace{0.3cm}
%
%
\hspace{0.8cm}
\begin{subfigure}[b]{0.43\textwidth}
       \centering
        \resizebox{\linewidth}{!}{
\begin{tikzpicture}
\def\vr{2pt}
\def\len{1}

\coordinate(a_1) at (0, 0); 
\coordinate(a_2) at (1, 0); 
\coordinate(a_3) at (2, 0); 
\coordinate(a_4) at (3, 0); 
\coordinate(a_5) at (4, 0); 
\coordinate(a_6) at (5, 0); 
\coordinate(a_7) at (4, 1); 
\coordinate(a_8) at (3, 1); 
\coordinate(a_9) at (2, 1); 
\coordinate(x) at (1, 1); 
\draw (a_1)--(a_2)--(a_3)--(a_4)--(a_5) -- (a_6);
\draw (a_5)--(a_7)--(a_8)--(a_9)--(x) -- (a_2);
\draw (a_1)--(a_9);
\foreach \i in {1, 2, 3, 4, 5, 6, 7, 8, 9}{
\draw (a_\i)[fill=white]circle(\vr);}
\draw (x)[fill=white]circle(\vr);
\draw(a_1)node[above]{\scriptsize{$a$}};
\draw(a_2)node[above left]{\scriptsize{$b$}};
\draw(a_3)node[above]{\scriptsize{$c$}};
\draw(a_4)node[above]{\scriptsize{$d$}};
\draw(a_5)node[above right]{\scriptsize{$e$}};
\draw(a_6)node[above]{\scriptsize{$f$}};
\draw(a_7)node[right]{\scriptsize{$g$}};
\draw(a_8)node[below]{\scriptsize{$h$}};
\draw(a_9)node[below]{\scriptsize{$i$}};
\draw(x)node[left]{\scriptsize{$j$}};
\end{tikzpicture}}
\caption{$H_8$}
\label{h8}
\end{subfigure}
%

%
%
\begin{subfigure}[b]{0.36\textwidth}
       \centering
        \resizebox{\linewidth}{!}{
\begin{tikzpicture}
\def\vr{2pt}
\def\len{1}

\coordinate(a_1) at (0, 0); 
\coordinate(a_2) at (1, 0); 
\coordinate(a_3) at (2, 0); 
\coordinate(a_4) at (3, 0); 
\coordinate(a_5) at (4, 0); 
\coordinate(a_6) at (1, 1); 
\coordinate(a_7) at (2, 1); 
\coordinate(a_8) at (3, 1); 
\draw (a_1)--(a_2)--(a_3)--(a_4)--(a_5);
\draw (a_2)--(a_6);
\draw (a_3)--(a_7);
\draw (a_4)--(a_8);
\foreach \i in {1, 2, 3, 4, 5, 6, 7, 8, 9}{
\draw (a_\i)[fill=white]circle(\vr);}
\draw(a_1)node[above]{\scriptsize{$a$}};
\draw(a_2)node[above left]{\scriptsize{$b$}};
\draw(a_3)node[above left]{\scriptsize{$c$}};
\draw(a_4)node[above left]{\scriptsize{$d$}};
\draw(a_5)node[above]{\scriptsize{$e$}};
\draw(a_6)node[left]{\scriptsize{$f$}};
\draw(a_7)node[left]{\scriptsize{$g$}};
\draw(a_8)node[left]{\scriptsize{$h$}};
\end{tikzpicture}}
\caption{$H_9$}
\label{h9}
\end{subfigure}

\caption{$4$-$\pch$-vertex-critical graphs}
\label{fig:4kriticni}
\end{figure}

%
%

\begin{theorem}
The graphs $K_4, H_1, H_2, H_3, H_4, H_5, H_6, H_7, H_8$ and $H_9$ shown in Fig.~\ref{fig:4kriticni} are $4$-$\pch$-vertex-critical. 
\label{izrek_4kriticni_grafi}
\end{theorem}

\begin{proof}
Theorem \ref{thm:kriticni_rall_klavzar_izrek} implies that $H_3, H_5$ and $H_7$ are $4$-$\pch$-vertex-critical. Clearly, also $K_4$ is $4$-$\pch$-vertex-critical. 

Now, consider the graph $H_1$ (see Fig.~\ref{Trikotnik_drugi}). If there exists a $3$-packing coloring $c_1$ of $H_1$, then $c_1(a), c_1(b),c_1(d) \neq 1$ and $a,b,d$ receive pairwise distinct colors by $c_1$. This implies that $c_1$ uses at least $4$ colors, a contradiction to $c_1$ being a $3$-packing coloring of $H_1$. Thus, $\pch(H_1) \geq 4$.
Next, color the vertices of $H_1$ as follows. Let $a$ and $c$ receive color $1$, and let the other vertices of $H_1$ receive distinct colors from $\{2,3,4\}$. In this way, a $4$-packing coloring of $H_1$ is formed. Thus, $\pch(H_1)=4$. 
Further, for any $u \in \{a,b,c,d,e\}$, $H_1-u$ has exactly $4$ vertices, but is not isomorphic to $K_4$, which implies that $\pch(H_1-u) \leq 3$. Hence, $H_1$ is a $4$-$\pch$-vertex-critical graph. 

Next, we prove that $H_2$ (see Fig.~\ref{Trikotnik_tretji}) is $4$-$\pch$-vertex-critical. Since each packing coloring of $H_2$ assigns three distinct colors to the vertices $a,b,e$ and (at most) one of these colors can be used for the vertices $c$ and $d$, we have $\pch(H_2) \geq 4$. 
On the other hand, there exists a $4$-packing coloring of $H_2$. Indeed, let $a$ and $c$ receive color $1$, and let the other vertices of $H_2$ receive different colors from $\{2,3,4\}$. Thus, $\pch(H_2)=4$. We observe that for any $u \in \{a,b,c,d,e\}$, $H_2-u$ has exactly $4$ vertices, but is not isomorphic to $K_4$. Hence, $\pch(H_2-u) \leq 3$ and $H_2$ is a $4$-$\pch$-vertex-critical graph. 

Further, consider the graph $H_4$ (see Fig.~\ref{h4}). Suppose that there exists a $3$-packing coloring $c_4$ of $H_4$. Then, $\{c_4(c),c_4(d)\}=\{2,3\}$. Without loss of generality we may assume that $c_4(c)=3$, $c_4(d)=2$. Then, $c_4(e)=1$ and there is no available color for $f$, a contradiction to our assumption. Thus, $\pch(H_4) \geq 4$. 
Now, color the vertices $a,b,c,d,e,f$ one after another using the colors $1,2,1,3,1,2$ and let $g$ receive a color $4$. This coloring is a $4$-packing coloring of $H_4$, which implies that $\pch(H_4)=4$. 
Next, let $u \in \{a,b,c,d,e,f,g\}$. If $u \in \{c,d,g\}$, then $H_4-u$ is a path or a union of paths. Consequently, $\pch(H_4-u) \leq 3$. If $u \in \{a,b\}$, then color the vertices $g$ and $e$ with color $1$, $c$ and $f$ with color $2$, $d$ with color $3$ and the other vertex with color $1$. This coloring is a $3$-packing coloring of $H_4-u$ and hence, $\pch(H_4-u) \leq 3$. In the case when $u \in \{e,f\}$, analogously we prove that $\pch(H_4-u) \leq 3$. Therefore, $H_4$ is $4$-$\pch$-vertex-critical. 

Next, we claim that $H_6$ (see Fig.~\ref{h6}) is a $4$-$\pch$-vertex-critical graph. If there exists a $3$-packing coloring $c_6$ of $H_6$, then $c_6(a), c_6(b),c_6(e) \neq 1$ and $a,b,e$ receive pairwise distinct colors by $c_6$. This implies that $c_6$ uses at least $4$ colors, a contradiction to $c_6$ being a $3$-packing coloring of $H_6$. Thus, $\pch(H_6) \geq 4$.
Further, by letting $c_6'(b)=c_6'(c)=c_6'(d)=1$, and by assigning different colors from $\{2,3,4\}$ to other vertices of $H_6$, we infer that $c_6'$ is a packing coloring of $H_6$ using $4$ colors. Thus, $\pch(H_6)=4$. 
Now, let $u \in \{a,b,c,d,e,f\}$. If $u \in \{a,e,f\}$, then $c_6'$ restricted to $H_6-u$ is a $3$-packing coloring of $H_6-u$. Hence, $\pch(H_6-u) \leq 3$. 
If $u=b$, then all vertices of $H_6-u$ can be colored with $3$ colors, since $H_6-u$ is the disjoint union of an isolated vertex and a cycle $C_4$. In the case when  $u=c$ (respectively, $u=d$), $H_6-u \cong Y_1$ and Lemma \ref{lema:trikotnik+pot=3barve} implies that $\pch(H_6-u) = 3$. Thus, $H_6$ is $4$-$\pch$-vertex-critical. 

Now, we prove that $H_8$ (see Fig.~\ref{h8}) is $4$-$\pch$-vertex-critical. Since $H_8$ contains a subgraph isomorphic to $H_7$ and $\pch(H_7)=4$, the hereditary property of the packing chromatic number implies that $\pch(H_8) \geq 4$. In order to show that $\pch(H_8)=4$, we form a $4$-packing coloring $c_8$ of $H_8$. Let $c_8(d)=c_8(i)=2$, $c_8(b)=c_8(g)=3$, $c_8(f)=4$ and let the remaining vertices receive color $1$. Since this is a $4$-packing coloring of $H_8$, $\pch(H_8)=4$.
Next, let $u \in \{a,b,c,d,e,f,g,h,i,j\}$ and let $G'$ be a subgraph of $H_8$ induced by $V(H_8) \setminus \{u\}$. 
Note that, if $u=f$, then $c_8$ restricted to $G'$ is a $3$-packing coloring of $G'$, thus $\pch(G') \leq 3$. Further, suppose that $u=e$. Let $c'_8(x)=c_8(x)$ for any $V(G') \setminus \{f\}$ and let $c'_8(f)=1$. Clearly, $c'_8$ is a $3$-packing coloring of $G'$, hence $\pch(G') \leq 3$.
If $u \in \{a,j\}$, then $G'$ is isomorphic to a subgraph of $H_7$ induced by $V(H_7) \setminus \{a\}$. Since $\pch(H_7-a) \leq 3$, we have $\pch(G') \leq 3$. 
Similarly, if $u \in \{b,i\}$, then $G'$ is isomorphic to a subgraph of $H_7$ induced by $V(H_7) \setminus \{i\}$. Since $\pch(H_7-i) \leq 3$, we have $\pch(G') \leq 3$. 
Now, let $u \in \{c,d\}$. In this case, color the vertices $e$ and $i$ with color $2$, $b$ and $g$ with color $3$ and the others with color $1$. Since such coloring is a $3$-packing coloring of $G'$, $\pch(G') \leq 3$. Analogously we prove that $\pch(G') \leq 3$ if $u \in \{h,g\}$. Therefore, $H_8$ is $4$-$\pch$-vertex-critical.

Finally, consider the graph $H_9$ (see Fig.~\ref{h9}). Suppose that there exists a $3$-packing coloring $c_9$ of $H_9$. Clearly, $c_9(b), c_9(c), c_9(d) \neq 1$, which implies that there is no available color for one vertex from $\{b,c,d\}$, a contradiction to $c_9$ being a $3$-packing coloring of $H_9$. This means that $\pch(H_9) \geq 4$. 
Next, we present a $4$-packing coloring $c_9'$ of $H_9$. Color the vertices $a,e,f,g,h$ with color $1$ and let $c_9'(b)=2$, $c_9'(c)=3$, $c_9'(d)=4$. Clearly, $c_9'$ is a $4$-packing coloring of $H_9$, thus $\pch(H_9)=4$. 
Further, let $u \in \{a,b,c,d,e,f,g,h\}$ and let $G'$ be a subgraph of $H_9$ induced by $V(H_9) \setminus \{u\}$. If $u \in \{b,c,d\}$, then $c_9'$ restricted to $G'$ is a $3$-packing coloring of $G'$. Hence, $\pch(G') \leq 3$. 
If $u=a$, then $G'$ can be colored using $3$ colors as follows. Let $b,e,g,h$ receive color $1$, $f$ and $d$ color $2$ and $c$ color $3$. Since this is a $3$-packing coloring of $G'$, $\pch(G') \leq 3$. If $u \in \{e,f,h\}$, the proof is analogous. Finally, let $u=g$. In this case, color the vertices $a,c,e,f,h$ with color $1$, and the remaining two vertices with colors $2$ and $3$. Clearly, this is a $3$-packing coloring of $G'$, which means that $\pch(G')\leq 3$. This completes the proof. 

\qed
\end{proof}

%
Next, we present five infinite families of $4$-$\pch$-vertex-critical graphs. Let $ab$ be an edge of a given graph $G$. The \textit{subdivision of $ab$ $k$-times} is obtained by removing the edge $ab$ from $G$, and adding $k$ new vertices, $a_1, a_2, \ldots, a_k$, and $k+1$ new edges, $aa_1, a_1a_2, \ldots, a_kb$, to $G$. 

The family $\mathcal{F}_1$ contains all graphs that can be constructed from the disjoint union of two copies of $K_3$ as follows. First, we make an edge joining both copies of $K_3$ and then, subdivide this edge $(4k)$-, $(4k+1)$- or $(4k+2)$-times, where $k \geq 0$. This family of graphs is shown in Fig.~\ref{Druzina_F1}.

Further, let $A$ be a copy of the complete graph $K_3$ and let $B$ be a copy of the path $P_3$. The family $\mathcal{F}_2$ (see Fig.~\ref{Druzina_F2}) contains all graphs that can be formed from the disjoint union of $A$ and $B$ by first joining a vertex $u\in V(A)$ with a non-pendant vertex $v \in V(B)$, and then subdividing the edge $uv$ $(4k+2)$-times, where $k$ is an arbitrary non-negative integer. In addition, also the graph with no subdivided edge $uv$ belongs to $\mathcal{F}_2$. 

The family $\mathcal{F}_3$ contains all graphs that can be obtained from the disjoint union of graphs $A$, where
$A \in \{P_4, C_4\}$, and $B$, where $B$ is isomorphic to $K_3$, by first joining a non-pendant vertex $u\in V(A)$ with a vertex $v \in V(B)$, and then subdividing the edge $uv$ $(4k)$-times, where $k$ is an arbitrary positive integer (see Fig.~\ref{Druzina_F3}). 

The graphs from $\mathcal{F}_4$ are formed from the disjoint union of two copies of a path $P_3$, denoted by $A$ and $B$, and one copy of a path $P_2$, denoted by $C$, in the following way. First, we join a non-pendant vertex $u \in V(A)$ with a vertex $v \in V(C)$, and a vertex $v \in V(C)$ with a non-pendant vertex $z \in V(B)$. Then, we subdivide the edge $uv$ $(4k+2)$-times and the edge $vz$ $(4k')$-times, where $k$ and $k'$ are two arbitrary non-negative integers (see Fig.~\ref{druzina_F4}). 
The vertices obtained by subdivision (if exist) are labeled by $y_1, y_2, \ldots, y_l$ and $w_1, w_2, \ldots, w_{l'}$ as in Fig.~\ref{druzina_F4}. In addition, a graph from $\mathcal{F}_4$ can also contain one edge from $\{v_1y_{l-1}, v_1w_2\}$.  

Finally, the family $\mathcal{F}_5$ contains all graphs, which can be obtained from the disjoint union of graphs $A$ and $B$, where
$A, B \in \{P_4, C_4\}$, by first joining a non-pendant vertex $u\in V(A)$ with a non-pendant vertex $v \in V(B)$, and then subdividing the edge $uv$ $(2k)$-times, where $k$ is an arbitrary non-negative integer (see Fig.~\ref{druzina_F5}).


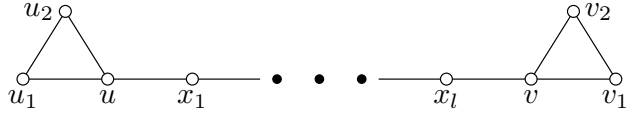
\begin{figure}[htb!]
%

%
%
\begin{subfigure}[b]{0.6\textwidth}
       \centering
        \resizebox{\linewidth}{!}{
\begin{tikzpicture}
\def\vr{2pt}
\def\len{1}
\coordinate (b) at (1, 0);
\coordinate (c) at (2, 0);
\coordinate (d) at (1.5, 0.8);
\coordinate (a') at (7, 0);
\coordinate (b') at (8, 0);
\coordinate (d') at (7.5, 0.8);

\coordinate (x) at (4, 0);
\coordinate (y) at (4.5, 0);
\coordinate (m) at (3, 0);
\coordinate (n) at (6, 0);
\coordinate (z) at (5, 0);

\draw (b) -- (c) -- (d); 
\draw (d')--(a') -- (b'); 
\draw (c)--(3,0)--(3.8,0);
\draw (a')--(6,0)--(5.2,0);
\draw (b)--(d);
\draw (b')--(d');

\draw (b)[fill=white]circle(\vr);
\draw (c)[fill=white]circle(\vr);
\draw (d)[fill=white]circle(\vr);
\draw (a')[fill=white]circle(\vr);
\draw (b')[fill=white]circle(\vr);
\draw (d')[fill=white]circle(\vr);
\draw (m)[fill=white]circle(\vr);
\draw (n)[fill=white]circle(\vr);

\draw (x)[fill=black]circle(0.05);
\draw (y)[fill=black]circle(0.05);
\draw (z)[fill=black]circle(0.05);


\draw(d)node[left]{\small{$u_2$}};
\draw(b)node[below]{\small{$u_1$}};
\draw(c)node[below]{\small{$u$}};

\draw(a')node[below]{\small{$v$}};
\draw(b')node[below]{\small{$v_1$}};
\draw(d')node[right]{\small{$v_2$}};

\draw(m)node[below]{\small{$x_1$}};
\draw(n)node[below]{\small{$x_{l}$}};
\draw(4.5, -0.5)node[below]{\small{$l \in \{4k, 4k+1, 4k+2;~ k \geq 0\}$}};

\end{tikzpicture}}
\caption{$\mathcal{F}_1$}
\label{Druzina_F1}
\end{subfigure}
\vspace{0.3cm}

%
%
\begin{subfigure}[b]{0.6\textwidth}
       \centering
        \resizebox{\linewidth}{!}{
\begin{tikzpicture}
\def\vr{2pt}
\def\len{1}
\coordinate (b) at (1, 0);
\coordinate (c) at (2, 0);
\coordinate (d) at (1.5, 0.8);
\coordinate (a') at (7, 0);
\coordinate (b') at (8, 0);
\coordinate (d') at (7.5, 0.8);

\coordinate (x) at (4, 0);
\coordinate (y) at (4.5, 0);
\coordinate (m) at (3, 0);
\coordinate (n) at (6, 0);
\coordinate (z) at (5, 0);

\draw (b) -- (c) -- (d); 
\draw (d')--(a') -- (b'); 
\draw (c)--(3,0)--(3.8,0);
\draw (a')--(6,0)--(5.2,0);
\draw (b)--(d);

\draw (b)[fill=white]circle(\vr);
\draw (c)[fill=white]circle(\vr);
\draw (d)[fill=white]circle(\vr);
\draw (a')[fill=white]circle(\vr);
\draw (b')[fill=white]circle(\vr);
\draw (d')[fill=white]circle(\vr);
\draw (m)[fill=white]circle(\vr);
\draw (n)[fill=white]circle(\vr);

\draw (x)[fill=black]circle(0.05);
\draw (y)[fill=black]circle(0.05);
\draw (z)[fill=black]circle(0.05);


\draw(4.5, -0.5)node[below]{\small{$l \in \{0, 4k+2; k \geq 0\}$}};

\draw(d)node[left]{\small{$u_2$}};
\draw(b)node[below]{\small{$u_1$}};
\draw(c)node[below]{\small{$u$}};

\draw(a')node[below]{\small{$v$}};
\draw(b')node[below]{\small{$v_1$}};
\draw(d')node[right]{\small{$v_2$}};

\draw(m)node[below]{\small{$x_1$}};
\draw(n)node[below]{\small{$x_{l}$}};

\end{tikzpicture}}
\caption{$\mathcal{F}_2$}
\label{Druzina_F2}
\end{subfigure}
\hspace{0.25cm}
\vspace{0.3cm}

%
%

\begin{subfigure}[b]{0.57\textwidth}
       \centering
        \resizebox{\linewidth}{!}{
\begin{tikzpicture}
\def\vr{2pt}
\def\len{1}
\coordinate (a) at (1, 1);
\coordinate (b) at (1, 0);
\coordinate (c) at (2, 0);
\coordinate (d) at (2, 1);
\coordinate (a') at (7, 0);
\coordinate (b') at (8, 0);
\coordinate (d') at (7.5, 0.8);

\coordinate (x) at (4, 0);
\coordinate (y) at (4.5, 0);
\coordinate (m) at (3, 0);
\coordinate (n) at (6, 0);
\coordinate (z) at (5, 0);

\draw (a)--(b) -- (c) -- (d); 
\draw [dashed] (a)--(d);
\draw (d')--(a') -- (b'); 
\draw (c)--(3,0)--(3.8,0);
\draw (a')--(6,0)--(5.2,0);
\draw (b')--(d');

\draw (a)[fill=white]circle(\vr);
\draw (b)[fill=white]circle(\vr);
\draw (c)[fill=white]circle(\vr);
\draw (d)[fill=white]circle(\vr);
\draw (a')[fill=white]circle(\vr);
\draw (b')[fill=white]circle(\vr);
\draw (d')[fill=white]circle(\vr);
\draw (m)[fill=white]circle(\vr);
\draw (n)[fill=white]circle(\vr);

\draw (x)[fill=black]circle(0.05);
\draw (y)[fill=black]circle(0.05);
\draw (z)[fill=black]circle(0.05);


\draw(4.5, -0.5)node[below]{$l \in \{4k;~ k \geq 1\}$};
\draw(a)node[left]{$u_2$};
\draw(b)node[left]{$u_1$};
\draw(c)node[below]{$u$};
\draw(d)node[right]{$u_3$};

\draw(a')node[below]{$v$};
\draw(b')node[below]{$v_1$};
\draw(d')node[right]{$v_2$};

\draw(m)node[below]{$x_1$};
\draw(n)node[below]{$x_{l}$};

\end{tikzpicture}}
\caption{$\mathcal{F}_3$}
\label{Druzina_F3}
\end{subfigure}
\vspace{0.3cm}


\begin{subfigure}[b]{1\textwidth}
       \centering
        \resizebox{\linewidth}{!}{
\begin{tikzpicture}
\def\vr{2pt}
\def\len{1}

\coordinate (a) at (0, 0);
\coordinate (b) at (1, 0);
\coordinate (c) at (1, 1);
\coordinate (d) at (6, 0);
\coordinate (e) at (7, 1);
\coordinate (f) at (13, 0);
\coordinate (g) at (14, 0);
\coordinate (h) at (13, 1);
\coordinate (m) at (2, 0);
\coordinate (n) at (5, 0);
\coordinate (w) at (7, 0);
\coordinate (w') at (12, 0);
\coordinate (s) at (8, 0);

\coordinate (x) at (3, 0);
\coordinate (y) at (3.5, 0);
\coordinate (z) at (4, 0);
\coordinate (x') at (10, 0);
\coordinate (y') at (10.5, 0);
\coordinate (z') at (11, 0);

\coordinate (t) at (9, 0);

\draw (a) -- (b) -- (c); 
\draw (w)--(e); 
\draw (g) -- (f) -- (h); 

\draw (b)--(2,0)--(2.8,0);
\draw (w)--(d)--(5,0)--(4.2,0);
\draw (f)--(11.2,0);
\draw (s)--(9.8,0);

\draw(w)--(s);

\draw[dashed](t)--(e)--(5,0);

\draw (a)[fill=white]circle(\vr);
\draw (b)[fill=white]circle(\vr);
\draw (c)[fill=white]circle(\vr);
\draw (d)[fill=white]circle(\vr);
\draw (e)[fill=white]circle(\vr);
\draw (f)[fill=white]circle(\vr);
\draw (g)[fill=white]circle(\vr);
\draw (h)[fill=white]circle(\vr);
\draw (m)[fill=white]circle(\vr);
\draw (n)[fill=white]circle(\vr);
\draw (w)[fill=white]circle(\vr);
\draw (w')[fill=white]circle(\vr);
\draw (s)[fill=white]circle(\vr);
\draw (t)[fill=white]circle(\vr);

\draw (x)[fill=black]circle(0.05);
\draw (y)[fill=black]circle(0.05);
\draw (z)[fill=black]circle(0.05);
\draw (x')[fill=black]circle(0.05);
\draw (y')[fill=black]circle(0.05);
\draw (z')[fill=black]circle(0.05);


\draw(3.5, -0.5)node[below]{$l \in \{4k+2;~ k \geq 0\}$ };
\draw(10.5, -0.5)node[below]{$l' \in \{4k';~ k' \geq 0\}$};

\draw(a)node[below]{$u_1$};
\draw(b)node[below]{$u$};
\draw(c)node[left]{$u_2$};

\draw(d)node[below]{$y_l$};
\draw(e)node[left]{$v_1$};

\draw(f)node[below]{$z$};
\draw(g)node[below]{$z_1$};
\draw(h)node[right]{$z_2$};

\draw(m)node[below]{$y_1$};
\draw(5,0)node[below]{$y_{l-1}$}; 
\draw(w)node[below]{$v$};
\draw(12,0)node[below]{$w_{l'}$}; 
\draw(s)node[below]{$w_{1}$}; 
\draw(t)node[below]{$w_{2}$}; 

\end{tikzpicture}}
\caption{$\mathcal{F}_4$}
\label{druzina_F4}
\end{subfigure}
\vspace{0.3cm}

%
%
%
\begin{subfigure}[b]{0.65\textwidth}
        \resizebox{\linewidth}{!}{
\begin{tikzpicture}
\def\vr{2pt}
\def\len{1}

\coordinate (a) at (0, 0);
\coordinate (b) at (1, 0);
\coordinate (c) at (2, 0);
\coordinate (d) at (2, 1);
\coordinate (a') at (7, 0);
\coordinate (b') at (8, 0);
\coordinate (c') at (9, 0);
\coordinate (d') at (7, 1);

\coordinate (x) at (4, 0);
\coordinate (y) at (4.5, 0);
\coordinate (m) at (3, 0);
\coordinate (n) at (6, 0);
\coordinate (z) at (5, 0);

\draw (a) -- (b) -- (c) -- (d); 
\draw (d')--(a') -- (b') -- (c'); 
\draw (c)--(3,0)--(3.8,0);
\draw (a')--(6,0)--(5.2,0);
\draw [dashed] (a) -- (d);
\draw [dashed] (c') -- (d');

\draw (a)[fill=white]circle(\vr);
\draw (b)[fill=white]circle(\vr);
\draw (c)[fill=white]circle(\vr);
\draw (d)[fill=white]circle(\vr);
\draw (a')[fill=white]circle(\vr);
\draw (b')[fill=white]circle(\vr);
\draw (c')[fill=white]circle(\vr);
\draw (d')[fill=white]circle(\vr);
\draw (m)[fill=white]circle(\vr);
\draw (n)[fill=white]circle(\vr);

\draw (x)[fill=black]circle(0.05);
\draw (y)[fill=black]circle(0.05);
\draw (z)[fill=black]circle(0.05);


\draw(4.5, -0.5)node[below]{$l \in \{2k; ~k \geq 0\}$};

\draw(a)node[below]{$u_2$};
\draw(b)node[below]{$u_1$};
\draw(c)node[below]{$u$};
\draw(d)node[left]{$u_3$};

\draw(a')node[below]{$v$};
\draw(b')node[below]{$v_1$};
\draw(c')node[below]{$v_2$};
\draw(d')node[right]{$v_3$};

\draw(m)node[below]{$x_1$};
\draw(n)node[below]{$x_{l}$};

\end{tikzpicture}}
\caption{$\mathcal{F}_5$}
\label{druzina_F5}
\end{subfigure}
%
\hspace{0.25cm}
\caption{Families of $4$-$\pch$-vertex-critical graphs}
\label{fig:vse_druzine_kriticne}
\end{figure}


\begin{theorem}
Let $G \in \mathcal{F}_1 \cup \mathcal{F}_2 \cup \mathcal{F}_3 \cup \mathcal{F}_4 \cup \mathcal{F}_5$. Then, $G$ is a $4$-$\pch$-vertex-critical graph. 
\label{izrek:druzine_so_4kriticne}
\end{theorem}

\begin{proof}
Let $G \in \mathcal{F}_1 \cup \mathcal{F}_2 \cup \mathcal{F}_3 \cup \mathcal{F}_4 \cup \mathcal{F}_5$ be an arbitrary graph. We claim that $G$ is $4$-$\pch$-vertex-critical. 

First, let $G \in \mathcal{F}_1$ with the vertices labeled as in Fig.~\ref{Druzina_F1}. 
Suppose that $\pch(G)=3$ (clearly, $\pch(G) \geq 3$, since $G$ contains a subgraph isomorphic to $K_3$). 
Then, $G \in \mathcal{G}_3$, which implies that $V(G)$ can be partitioned into the sets $V_0, V_1, \ldots, V_7$, defined in Section 2. Since $\deg(u)=\deg(v)=3$,  $u$ and $v$ belong to a subgraph of $G$ isomorphic to $K_3$ and have neighbors, which are neither leaves nor belong to some triangle, we infer that $u, v \in V_3$. 
The fact that $u \in V_3$ implies that $x_1 \in V_2, x_2 \in V_1, x_3 \in V_2, x_4 \in V_3, \ldots$ Recall that any two vertices from $V_3$ are at distance $4a, a\geq 1$. Therefore, if $l \leq 2$, we have a contradiction, since $d(u,v) \leq 3$ and $u,v \in V_3$.  Thus, $l \geq 3$, but then there exists $ c \in \{x_{l-2}, x_{l-1}, x_l\}$, which belongs to $V_3$. Since $d(c,v) \leq 3$ and $c,v \in V_3$, we again have a contradiction. Hence, $G \notin \mathcal{G}_3$. 
In order to prove that $\pch(G) \leq 4$, we form a $4$-packing coloring $c_1$ of $G$. Let $c_1(u)=3$, $c_1(u_1)=c_1(v_1)=1$, $c_1(u_2)=2$, $c_1(v)=4$. Further, color the vertices $x_1, x_2, \ldots, x_l$ one after another using the following pattern of colors: $1,2,1,3$. Finally, let $c_1(v_2)=2$ if $l=4k$ or $l=4k+1$ for any $k \geq 0$, and otherwise, $c_1(v_2)=3$. Since $c_1$ is a $4$-packing coloring of $G$, $\pch(G)=4$.
Now, let $a \in V(G)$ be an arbitrary vertex. Note that $G-a$ is either $X_n$, $n \geq 2$, or the disjoint union of two graphs from $\{K_2, K_3\} \cup \{X_m; m \geq 1\}$. Using Lemma \ref{lema:trikotnik+pot=3barve} and the facts that $\pch(K_2)=2$, $\pch(K_3)=3$, we infer that $\pch(G-a) \leq 3$. Thus, $G$ is $4$-$\pch$-vertex-critical.

Next, suppose that $G \in \mathcal{F}_2$. Let the vertices of $G$ be denoted as is shown in Fig.~\ref{Druzina_F2}. We claim that $\pch(G)=4$. Note that $G$ is isomorphic to a subgraph of some graph $H$ from $\mathcal{F}_1$. We have already proved that $\pch(H)=4$, which implies that $\pch(G) \leq 4$. Therefore, we need to prove that $\pch(G) \geq 4$. Suppose to the contrary that $\pch(G)=3$ (clearly, $\pch(G) \geq 3$, since $G$ contains a subgraph isomorphic to $K_3$). Then, $G \in \mathcal{G}_3$, which means that $V(G)$ can be partitioned into the sets $V_0, V_1, \ldots, V_7$, defined above. Since $\deg(u)=3$,  $u$ belongs to a subgraph of $G$ isomorphic to $K_3$ and has neighbors, which are neither leaves nor belong to some triangle, we derive that $u \in V_3$. If $l \neq 0$, then  $x_1 \in V_2, x_2 \in V_1, x_3 \in V_2, x_4 \in V_3, \ldots, x_l \in V_1$ and consequently, $v \in V_0 \cup V_2$, a contradiction since $v$ has degree $3$. Otherwise, $u_1, u_2 \in V_5$, which implies that $v \in V_2$, again a contradiction since $\deg(v)=3$. Therefore, in both cases, $G \notin \mathcal{G}_3$, which implies that $\pch(G)=4$. 
Now, we need to prove that $\pch(G-a) \leq 3$ holds for any $a \in V(G)$. Let $a \in V(G)$ be an arbitrary vertex. We observe that $G-a$ is a subgraph of a graph $H-b$, where $H \in \mathcal{F}_1$ and $b \in V(H)$. We have already proved $\pch(H-b) \leq 3$, which implies that $\pch(G-a) \leq 3$. Thus, $G$ is a $4$-$\pch$-vertex-critical graph.

Now, let $G \in \mathcal{F}_3$ and let its vertices be denoted as is shown in Fig.~\ref{Druzina_F3}. We claim that $\pch(G)=4$. It is clear that $\pch(G) \geq 3$, since $G$ contains a subgraph isomorphic to $K_3$. Suppose that $\pch(G) = 3$, which implies that $G \in \mathcal{G}_3$ and hence, $V(G)$ can be partitioned into the sets $V_0, V_1, \ldots, V_7$, defined above. 
Since $\deg(v)=3$, $v$ belongs to a subgraph of $G$ isomorphic to $K_3$ and has a neighbor, which is neither leaf nor belongs to some triangle, we conclude that $v \in V_3$. This implies that $x_l \in V_2, x_{l-1} \in V_1, x_{l-2} \in V_2, x_{l-3} \in V_3, \ldots, x_1 \in V_3$. Then, $u \in V_2 \cup V_4 \cup V_5 \cup V_6$, more precisely $u \in V_6$ since $\deg(u)=3$ and $u$ does not belong to any triangle. Further, since $u_1, u_3$ do not belong to any subgraph of $G$ isomorphic to $K_3$, they belong to $V_7$, a contradiction, since $u_1$ is not a leaf (recall that $V_7$ contains only the leaves). Thus, $\pch(G) \geq 4$. 
Next, note that $G-u_3$ is isomorphic to $X_n$, $n \geq 1$. Lemma \ref{lema:trikotnik+pot=3barve} implies that there exists a $3$-packing coloring $c_3'$ of $G-u_3$. Let $c_3(a)=c_3'(a)$ for any $a \in V(G) \setminus \{u_3\}$ and let $c_3(u_3)=4$. Clearly, $c_3$ is a $4$-packing coloring of $G$, thus $\pch(G)=4$. 
Now, we prove that $\pch(G-a) \leq 3$ holds for any $a \in V(G)$. If $a=u_3$, then $G-a$ is isomorphic to $X_n$, $n \geq 1$. Next, if $a \in \{v_1,v_2\}$, then $G-a$ is isomorphic to (a subgraph of) $Y_n, n \geq 1$. In the case when $a \in V(G) \setminus \{u_2, u_3, v_1, v_2\}$, $G-a$ is isomorphic to the disjoint union of (a subgraph of) $X_n$, $n \geq 1$, and (a subgraph of) $Y_n$, $n \geq 1$. In each case, Lemma \ref{lema:c4+pot=3barve} implies that $\pch(G-a) \leq 3$. If $a=u_2$, then we form a $3$-packing coloring $c_3''$ of $G-u_2$ as follows.  Let $c_3''(u_1)=c_3''(u_3)=c_3''(v_1)=1$, $c_3''(u)=c_3''(v_2)=2$, $c_3''(v)=3$ and color the vertices $x_1, x_2, \ldots, x_l$ one after another using the following pattern of colors: $3,1,2,1$. Since $c_3''$ is a $3$-packing coloring of $G-u_2$, $\pch(G-u_2) \leq 3$. Thus, $G$ is $4$-$\pch$-vertex-critical.

Further, let $G\in \mathcal{F}_4$ with the vertices labeled as in Fig.~\ref{druzina_F4}. 
First, we prove that $\pch(G)=4$. Since $G$ contains a subgraph isomorphic to $P_4$, $\pch(G) \geq 3$. Suppose that $\pch(G)=3$, which means that $G \in \mathcal{G}_3$. Then, we can partition $V(G)$ into the sets $V_0, V_1, \ldots, V_7$, described above. Since $\deg(v)=3$ and $v$ has two neighbors, which are neither leaves nor belong to some triangle, $v \in V_1 \cup V_3$. Then, $y_l \in V_2, y_{l-1} \in V_1 \cup V_3, y_{l-2} \in V_2, y_{l-3} \in V_1 \cup V_3, \ldots, y_1 \in V_1 \cup V_3$. Further, since $\deg(u)=3$, $u \in V_6$ and it follows that $v \in V_1$. Hence, $w_1 \in V_2$, $w_2 \in V_3$, $w_3 \in V_2$, $w_4 \in V_1, \ldots, w_{l'} \in V_1$. This means that $z \in V_0 \cup V_2$, a contradiction since $\deg(z)=3$. Thus, $G \notin \mathcal{G}_3$, which implies that $\pch(G) \geq 4$.
Now, let $c_4: V(G) \longrightarrow \{1,2,3,4\}$ and let $c_4(u_1)=c_4(u_2)=c_4(v_1)=c_4(z_1)=c_4(z_2)=1$, $c_4(u)=c_4(v)=2$, and $c_4(z)=4$. Further, color the vertices $y_1, y_2, \ldots, y_l$ one after another using the pattern of colors $3,1,2,1$, and the vertices $w_1, w_2, \ldots, w_{l'}$ one after another using the pattern of colors $1,3,1,2$. Clearly, $c_4$ is a $4$-packing coloring of $G$. Thus, $\pch(G) = 4$. 
Next, we claim that $\pch(G-a) \leq 3$ holds for any $a \in V(G)$. 
Denote by $G'$ a subgraph of $G$ induced by $V(G) \setminus \{z_2\}$. Let $c_4'(u_1)=c_4'(u_2)=c_4'(v_1)=c_4'(z)=1$, $c_4'(u)=2$, $c_4'(z_1)=3$. Then, color the vertices $y_1, y_2, \ldots, y_l, v, w_1, w_2, \ldots, w_{l'}$ one after another using the following pattern of colors: $3,1,2,1$. Clearly, $c_4'$ is a $3$-packing coloring of $G'$, thus $\pch(G-z_2) \leq 3$. Analogously we prove that $\pch(G-z_1) \leq 3$.
Let $G''$ be a subgraph of $G$ induced by $V(G) \setminus \{u_2\}$. In order to prove that $\pch(G'') \leq 3$, we form a $3$-packing coloring $c_4''$ of $G''$. 
Let $c_4''(u)=c_4''(v_1)=c_4''(z_1)=c_4''(z_2)=1$, $c_4''(z)=2$, $c_4''(u_1)=3$. Then, color the vertices $y_1, y_2, \ldots, y_l, v, w_1, w_2, \ldots, w_{l'}$ one after another using the following pattern of colors: $2,1,3,1$. Clearly, $c_4''$ is a $3$-packing coloring of $G''$, thus $\pch(G-u_2) \leq 3$. Analogously we prove that $\pch(G-u_1) \leq 3$.
If $a \in \{u, y_1,y_2, \ldots, y_{l-1}, w_2, \ldots, w_{l'}, z, v\}$, then $G-a$ is the disjoint union of a subgraph of $G'$ and a subgraph of $G''$. Thus, the above considerations imply that $\pch(G-a) \leq 3$. 
In the case when $a=v_1$, color the vertices of $G-v_1$ as follows. Let $u_1, u_2, z_1, z_2$ receive color $1$, $u$ and $z$ color $3$, and let the vertices $y_1, y_2, \ldots, y_l, v, w_1, w_2, \ldots, w_{l'}$ be colored one after another using the pattern $1,2,1,3$. The described coloring is a $3$-packing coloring of $G-v_1$, thus $\pch(G-v_1) \leq 3$. 
Finally, if $a \in \{y_l, w_1\}$, then $G-a$ is isomorphic to $G-v_1$, or to the disjoint union of a subgraph of $G'$ and a subgraph of $G''$. Hence, $\pch(G-a) \leq 3$ and we conclude that $G$ is $4$-$\pch$-vertex-critical.

It remains to consider the case when $G \in \mathcal{F}_5$. Denote the vertices of $G$ as is shown in Fig.~\ref{druzina_F5}. First, we prove that $\chi_\rho(G)=4$.
It is clear that $\pch(G) \geq 3$ since $G$ contains a subgraph isomorphic to $P_4$. Suppose that $G \in \mathcal{G}_3$. Then, we make a partition of $V(G)$ into the sets $V_0, V_1, \ldots, V_7$, which are described above. Since $\deg(u), \deg(v) \geq 3$ and $d(u,v)=2k+1$, it follows that $u$ or $v$ belongs to $V_6$. But this is a contradiction, since $u$ (respectively, $v$) has at least two neighbors, which are neither leaves nor belong to a triangle. Thus, $\pch(G) \geq 4$. 
Let $c_5: V(G) \longrightarrow \{1,2,3,4\}$ and let $c_5(u_1)=c_5(u_3)=c_5(v_1)=c_5(v_3)=1$, $c_5(u_2)=c_5(v_2)=2$, $c_5(u)=3$ and $c_5(v)=4$. Further, color the vertices $x_1, x_2, \ldots, x_l$ one after another using the pattern of colors: $1,2,1,3$. Clearly, $c_5$ is a $4$-packing coloring of $G$ and thus, $\pch(G) = 4$. 
Next, we prove that $\pch(G-a) \leq 3$ holds for any $a \in V(G)$. If $a \in V(G) \setminus \{u_2,v_2\}$, then $G-a$ is isomorphic to (a subgraph of) $Y_n$, $n \geq 1$,  or it is the disjoint union of graphs isomorphic to (subgraphs of) $Y_n$, $n \geq 1$. In each case, Lemma \ref{lema:c4+pot=3barve} implies that $\pch(G-a) \leq 3$. 
If $a=v_2$ and $l=4k', k' \geq 0$, then we form a $3$-packing coloring $c_5'$ of $G-v_2$ as follows: $c_5'(a)=c_5(a)$ for any $a \in V(G) \setminus \{v,v_2\}$ and $c_5'(v)=2$. In the case when $a=v_2$ and $l=4k'+2, k' \geq 0$, let $c_5''(u_1)=c_5''(u_3)=c_5''(v_1)=c_5''(v_3)=1$, $c_5''(u_2)=3$, $c_5''(u)=c_5''(v)=2$, and color the vertices $x_1, x_2, \ldots, x_l$ one after another using the pattern of colors: $1,3,1,2$. These colorings imply that  $\pch(G-v_2) \leq 3$. Analogously one can prove that  $\pch(G-u_2) \leq 3$. Hence, $G$ is a $4$-$\pch$-vertex-critical graph. This completes the proof. 

\qed
\end{proof}


The following two lemmas will help us to prove a complete characterization of $4$-$\pch$-vertex-critical graphs. 

\begin{lemma}
Let $H$ be obtained by attaching a single leaf to two vertices of $C_n$, $n \geq 4$, which are at distance $2k+1$, $k \geq 0$. Then, $\pch(H) \geq 4$. 
\label{lemma:cikel+2lista_na_lihi_razdalji_pch_vsaj_4}
\label{lema_sosednji_vozlisci_stopnje_3_cikel_vsaj4barve}
\end{lemma}

\begin{proof}
Let $H$ be obtained by attaching a single leaf to two vertices of $C_n$, $n \geq 4$, which are at distance $2k+1$, $k \geq 0$. 
Clearly, if $n \neq 4a, a\geq 1$, then $H$ contains a subgraph $H'$ isomorphic to a graph from $\mathcal{C}$. Since $\pch(H')=4$, the hereditary property ot the packing chromatic number implies that $ \pch(H) \geq 4$. 
Therefore, we only need to consider the case when $n = 4a, a\geq 1$. If $a\geq 2$, then $H$ contains a subgraph isomorphic to a graph from $\mathcal{F}_5$. In this case, the hereditary property of the packing chromatic number and Theorem \ref{izrek:druzine_so_4kriticne} imply that $\pch(H) \geq 4$. If $a=1$, then $H$ is isomorphic to $H_5$ and by Theorem \ref{izrek_4kriticni_grafi}, $\pch(H)=4$. 
\qed
\end{proof}

\begin{lemma}
Let $\mathcal{X}=\mathcal{F}_1 \cup \mathcal{F}_2 \cup \mathcal{F}_3 \cup \mathcal{F}_4 \cup \mathcal{F}_5 \cup \mathcal{C}_5 \cup \mathcal{C}_6 \cup \mathcal{C} \cup \{K_4, H_1, H_2, H_3, H_4, H_5, H_{6}, H_7, \\ H_8, H_9\}$ and let $H \in \mathcal{X}$. If a graph $G \notin \mathcal{X}$ contains a subgraph isomorphic to $H$, then $G$ is not $4$-$\pch$-vertex-critical. 
\label{lemma:podgrafi_4-kriticnost}
\end{lemma}

\begin{proof}
Let $H \in \mathcal{X}=\mathcal{F}_1 \cup \mathcal{F}_2 \cup \mathcal{F}_3 \cup \mathcal{F}_4 \cup \mathcal{F}_5 \cup \mathcal{C}_5 \cup \mathcal{C}_6 \cup \mathcal{C} \cup \{K_4, H_1, H_2, H_3, H_4, H_5, H_6, H_7,\\ H_8, H_9\}$ and let $G \notin \mathcal{X}$ be an arbitrary graph containing a subgraph isomorphic to $H$. Clearly, since $G \notin \mathcal{X}$, but $H \in \mathcal{X}$, $G \not \cong H$. 

First, suppose that $V(G) \neq V(H)$, which means that there exists a vertex $a \in V(G) \setminus V(H)$. Using the fact that $G-a$ contains a subgraph isomorphic to $H$, the hereditary property of the packing chromatic number   and Theorems \ref{thm:kriticni_CIKLI_rall_klavzar_izrek}, \ref{izrek_4kriticni_grafi}, \ref{izrek:druzine_so_4kriticne},  we derive that $\pch(G-a) \geq 4$. It follows that $G$ is not a $4$-$\pch$-vertex-critical graph.

Now, let $V(G)=V(H)$. 
If $H \in \mathcal{C}_5 \cup \mathcal{C}_6 \cup \mathcal{C}$, then Theorem \ref{thm:kriticni_CIKLI_rall_klavzar_izrek} implies that $G$ is not $4$-$\pch$-vertex-critical (recall that $G \notin \mathcal{C}_5 \cup \mathcal{C}_6 \cup \mathcal{C}$). 
Further, if $H \cong K_4$, then $G \cong H$, a contradiction to the fact that $G \not \cong H$.

Clearly, since $G \not \cong H$ and $V(G)=V(H)$, $E(H)\neq E(G)$. Then, there exists $e' \in E(G) \setminus E(H)$. 

Therefore, if $H \cong H_1$ (see Fig.~\ref{Trikotnik_drugi}), then at least one of the edges $be, ca, ce, cd$ belongs to $E(G)$. This implies that $G$ contains a subgraph isomorphic to a graph from $\mathcal{C}$ or it contains a subgraph isomorphic to $K_4$. The above considerations imply that $G$ is not a $4$-$\pch$-vertex-critical graph.
If $H \cong H_2$ (see Fig.~\ref{Trikotnik_tretji}), then $G$ contains a subgraph isomorphic to a graph from $\mathcal{C}$ and by the above reasoning, it is not $4$-$\pch$-vertex-critical.
In the case when $H \cong H_3$ or $H \cong H_6$ (see Fig.~\ref{Trikotnik_cetrti},~\ref{h6}), $G$ contains a subgraph isomorphic to a graph from $\mathcal{C}$ or a subgraph isomorphic to $H_1$. In both cases, the above considerations imply that $G$ is not $4$-$\pch$-vertex-critical. 
Now, let $H \cong H_4$ (see Fig.~\ref{h4}). In this case, $G$ contains a subgraph isomorphic to a graph from $\mathcal{C} \cup \{H_1, H_2\}$. Again, the above considerations imply that $G$ is not $4$-$\pch$-vertex-critical. 
If $H \cong H_5$ (see Fig.~\ref{h5}), then $G$ contains a subgraph isomorphic to a graph from $\mathcal{C} \cup \{H_1\}$ or $G \cong H_6$. In the first case, the above considerations imply that $G$ is not $4$-$\pch$-vertex-critical, and in the second case we have a contradiction to $G \notin \mathcal{X}$.
Next, let $H \cong H_7$ (see Fig.~\ref{h7}). If $e' \in E(G) \setminus \{ad, ai, cf, fh\}$, then $G$ contains a subgraph isomorphic to a graph from $\mathcal{C}$ and the above consideration implies that $G$ is not $4$-$\pch$-vertex-critical. Further, if $ad$ or $cf$ belongs to $E(G)$, then $G$ contains a non-spanning subgraph isomorphic to a graph from $\mathcal{F}_5$ (the vertices from $V(G)$ corresponding to the vertices of degree $3$ in a graph from $\mathcal{F}_5$ are $d,e$ or $b,c$). By Theorem \ref{izrek:druzine_so_4kriticne}, $G$ is not $4$-$\pch$-vertex-critical. Next, if at least one of the edges from $\{ai,fh\}$ belongs to $E(G)$, then $G \cong H_8$ or $G$ contains a non-spanning subgraph isomorphic to a graph from $\mathcal{F}_5$. In the first case, we have a contradiction to our assumption, and in the second case we derive that $G$ is not $4$-$\pch$-vertex-critical. If $H \cong H_8$ (see Fig.~\ref{h8}), then we analogously prove that $G$ is not $4$-$\pch$-vertex-critical (note that $H_7$ is a subgraph of $H_8$). 
Futher, consider the case when $H \cong H_9$ (see Fig.~\ref{h9}). It is easy to observe that $G$ contains a non-spanning subgraph isomorphic to a graph from $\{H_3, H_5, C_5\} \cup \mathcal{F}_2$. Hence, using the fact that $\pch(C_5)=4$ and Theorems \ref{izrek_4kriticni_grafi}, \ref{izrek:druzine_so_4kriticne}, we deduce that $G$ is not $4$-$\pch$-vertex-critical.

Next, suppose that $H \in \mathcal{F}_1$ (see Fig.~\ref{Druzina_F1}).
If $e'=u_1a$, $a \in \{v,v_1,v_2\} \cup \{x_1, x_2, \ldots, x_l\}$, then $G$ contains a (non-spanning) subgraph isomorphic to $H_1$ or a (non-spanning) subgraph isomorphic to a graph from $\mathcal{C}$ ($G$ contains a  subgraph obtained by attaching a vertex to two adjacent vertices of a cycle of order at least $4$, hence, by Lemma \ref{lemma:cikel+trikotnik=neustrezen_cikel}, $G$ also contains a subgraph isomorphic to a graph from $\mathcal{C}$). In both cases, the above consideration implies that $G$ is not a $4$-$\pch$-vertex-critical graph, if $\deg(u_1) \geq 3$. Analogously we prove that $G$ is not $4$-$\pch$-vertex-critical, if $\deg(u_2), \deg(v_1), \deg(v_2) \geq 3$. 
Now, let $e'=ab$, $a,b \in \{u,v\} \cup \{x_1, x_2, \ldots, x_l\}$. Then, $G$ contains a (non-spanning) subgraph isomorphic to $H_2$, a subgraph isomorphic to $H_4$ or a non-spanning subgraph obtained by attaching a single leaf to two adjacent vertices of a cycle of order at least $4$. In each case, the above reasoning and Lemma \ref{lemma:cikel+2lista_na_lihi_razdalji_pch_vsaj_4} imply that $\pch(G-v_2) \geq 4$. Consequently, $G$ is not $4$-$\pch$-vertex-critical. 

%
%
Now, let $H \in \mathcal{F}_2$ (see Fig.~\ref{Druzina_F2}). 
If $e'=u_1a$ (or $e'=u_2a$), $a \in \{v,v_1,v_2\} \cup \{x_1, x_2, \ldots, x_l\}$, then analogously as in the case when $H \in \mathcal{F}_1$, we prove that $G$ is not $4$-$\pch$-vertex-critical, if $\deg(u_1) \geq 3$ or $\deg(u_2) \geq 3$. 
If $e'=ab$, $a,b \in \{u\} \cup \{x_1, x_2, \ldots, x_l\}$, then again, analogously as in the case when $H \in \mathcal{F}_1$, we prove that $G$ is not $4$-$\pch$-vertex-critical. 
Next, $e'=uv$ implies that $G$ contains a non-spanning subgraph obtained by attaching a single leaf to two adjacent vertices of $C_n, n \geq 4$. Hence, Lemma \ref{lema_sosednji_vozlisci_stopnje_3_cikel_vsaj4barve} implies that $G$ is not $4$-$\pch$-vertex-critical.  
If $e'=uv_1$ or $e'=uv_2$, then $G$ contains a subgraph isomorphic to a graph from $\{H_2\} \cup \mathcal{C}$, hence the above reasoning yields that $G$ is not $4$-$\pch$-vertex-critical. 
Now, let $e'=vx_i$, $i \in \{1,2, \ldots, l\}$. Then, $G$ contains a non-spanning subgraph isomorphic to a graph from $\mathcal{F}_1$ or a non-spanning subgraph obtained by attaching a single leaf to two adjacent vertices of $C_n, n \geq 4$. Hence, Theorem \ref{izrek:druzine_so_4kriticne} and Lemma \ref{lema_sosednji_vozlisci_stopnje_3_cikel_vsaj4barve}  imply that $G$ is not a $4$-$\pch$-vertex-critical graph. 
If $e'=v_1v_2$, then $G \in \mathcal{F}_1$ or it contains a subgraph isomorphic to a graph from $\mathcal{F}_1$. We have a contradiction to our assumption or the above consideration yields that $G$ is not $4$-$\pch$-vertex-critical.
Finally, let $e'=v_1x_i$ or $e'=v_2x_i$, $i \in \{1,2, \ldots, l\}$. Then, $G$ contains a non-spanning subgraph isomorphic to a graph from $\mathcal{C} \cup \mathcal{F}_1 \cup \mathcal{F}_2 \cup \mathcal{F}_3$. The above consideration and Theorem \ref{izrek:druzine_so_4kriticne} imply that $G$ is not  $4$-$\pch$-vertex-critical. 
%

%
%
%
Next, let $H \in \mathcal{F}_3$ (see Fig.~\ref{Druzina_F3}). 
If $e' \in \{v_1a, v_2a, bc\}$, where  $a \in \{u, u_1, u_2, u_3\} \cup \{x_1, x_2, \ldots, x_l\}$, $b,c \in \{u, v \} \cup \{x_1, x_2, \ldots, x_l\}$, then analogously as in the case when $H \in \mathcal{F}_1$, we prove that $G$ is not $4$-$\pch$-vertex-critical. Further, if $e'=u_1v$ or $e'=u_3v$, then $G$ contains a non-spanning subgraph isomorphic to a graph from $\mathcal{C}$, so it is not $4$-$\pch$-vertex-critical. 
Also in the case when $e'=vu_2$, $G$ is not $4$-$\pch$-vertex-critical. Namely, $G$ contains a non-spanning subgraph obtained by attaching a single leaf to two vertices of $C_n, n \geq 4$, which are at distance $3$, and Lemma \ref{lemma:cikel+2lista_na_lihi_razdalji_pch_vsaj_4} implies the result.
Next, we observe that $e'=uu_2$ or $e'=u_1u_3$ imply that $G$ contains a non-spanning subgraph isomorphic to a graph from $\mathcal{F}_1$. Hence, by Theorem \ref{izrek:druzine_so_4kriticne}, $G$ is not a $4$-$\pch$-vertex-critical graph. 
Now, suppose that $e'=u_1x_i$, $i \in \{1,2, \ldots, l \}$. Then, $G$ contains a non-spanning subgraph ismorphic to $H_3$ or a non-spanning subgraph obtained by attaching a single leaf to two adjacent vertices of a cycle of order at least $4$. By Theorem \ref{izrek_4kriticni_grafi} and Lemma \ref{lemma:cikel+2lista_na_lihi_razdalji_pch_vsaj_4}, $G$ is not $4$-$\pch$-vertex-critical. 
If $e'=u_2x_i$, $i \in \{1,2, \ldots, l \}$, then $G$ contains a non-spanning subgraph ismorphic to a graph from $\{H_5\} \cup \mathcal{C}$ or a non-spanning subgraph obtained by attaching a single leaf to two vertices of $C_n, n \geq 4$, which are at distance $3$.
Hence, by Theorems \ref{thm:kriticni_CIKLI_rall_klavzar_izrek} and \ref{izrek_4kriticni_grafi}, and Lemma \ref{lemma:cikel+2lista_na_lihi_razdalji_pch_vsaj_4}, $G$ is not $4$-$\pch$-vertex-critical. 
Finally, if $e'=u_3x_i$, $i \in \{1,2, \ldots, l \}$, then $G$ contains a subgraph ismorphic to a graph from $\{H_4\} \cup \mathcal{C} \cup \mathcal{F}_2$. From the above consideration, we derive that $G$ is not a $4$-$\pch$-vertex-critical graph. 
%
%

Further, suppose that $H \in \mathcal{F}_4$ with the vertices labeled as in Fig.~\ref{druzina_F4}.
If $e' \in \{u_1y_1, u_2y_1, u_1u_2, v_1w_1, z_1w_{l'}, z_2w_{l'}\}$, then $G$ contains a subgraph isomorphic to a graph from $\mathcal{F}_2$. By the previous consideration, $G$ is not $4$-$\pch$-vertex-critical.
Further, suppose that $e' \in \{u_1v, u_1v_1, u_1z, u_1z_1, u_1z_2, u_2v, u_2v_1, u_2z, u_2z_1, u_2z_2, uv_1, uz, uz_1, uz_2, vz, vz_1, \\ vz_2, v_1z\}$. In this case, $G$ contains a subgraph isomorphic to a cycle from $\mathcal{C}$. Thus, by the above reasoning, $G$ is not  $4$-$\pch$-vertex-critical.
If $e' \in \{v_1z_1, v_1z_2, uv\} \cup \{uw_{i'}; 1 \leq i' \leq l'\}$, then $G$ contains a non-spanning subgraph obtained by attaching a single leaf to two vertices of $C_n, n \geq 4$, which are at distance $2k+1, k \geq 0$. Lemma \ref{lemma:cikel+2lista_na_lihi_razdalji_pch_vsaj_4} implies that $G$ is not a $4$-$\pch$-vertex-critical graph. 
Next, let $e'=ab$, $a,b \in \{v, y_i, w_{i'}; 1 \leq i \leq l, 1 \leq i' \leq l'\}$. In this case, $G$ contains a non-spanning subgraph isomorphic to $H_4$ or a non-spanning subgraph obtained by attaching a single leaf to two adjacent vertices of $C_n, n \geq 4$. Using Theorem \ref{izrek_4kriticni_grafi} and Lemma \ref{lemma:cikel+2lista_na_lihi_razdalji_pch_vsaj_4}, we derive that $G$ is not $4$-$\pch$-vertex-critical. 
Further, suppose that $e' \in \{u_1y_i, u_2y_i, z_1w_{i'}, z_2w_{i'}, u_1w_{j'}, u_2w_{j'}, z_1y_j, z_2y_j;~2 \leq i \leq l, 1 \leq i' \leq l'-1, 1 \leq j' \leq l', 1 \leq j \leq l\}$. Then, $G$ contains a non-spanning subgraph isomorphic to a graph from $\mathcal{F}_5 \cup \mathcal{C}$ or a non-spanning subgraph obtained by attaching a single leaf to two adjacent vertices of $C_n, n \geq 4$ (if $l'=0$). Hence, by the above consideration, Theorem \ref{izrek:druzine_so_4kriticne} and Lemma \ref{lemma:cikel+2lista_na_lihi_razdalji_pch_vsaj_4}, we infer that $G$ is not $4$-$\pch$-vertex-critical. 
Now, let $e'=z_1z_2$, $e'=uy_i$ or $e'=zw_{i'}$, where $1 \leq i \leq l$, $1 \leq i' \leq l'$. This implies that $G$ contains a subgraph isomorphic to a graph from $\mathcal{F}_2 \cup \mathcal{F}_3$ or a non-spanning subgraph obtained by attaching a single leaf to two adjacent vertices of $C_n, n \geq 4$. Therefore, the above reasoning and Lemma \ref{lemma:cikel+2lista_na_lihi_razdalji_pch_vsaj_4} imply that $G$ is not $4$-$\pch$-vertex-critical. 
If $e'=zy_i$, $1 \leq i \leq l$, then $G$ contains a non-spanning subgraph isomorphic to a graph $H_3$ or a non-spanning subgraph obtained by attaching a single leaf to two adjacent vertices of $C_n, n \geq 4$. In both cases, we derive that $G$ is not $4$-$\pch$-vertex-critical. 
Finally, let $e'=v_1y_i$ or $e=v_1w_{i'}$, $i \in \{1,2, \ldots, l-2, l\}$, $i' \in \{3, 4, \ldots, l'\}$. Then, $G$ contains a non-spanning subgraph isomorphic to a graph from $\{H_4\} \cup  \mathcal{C} \cup \mathcal{F}_4$. The above consideration implies that $G$ is not $4$-$\pch$-vertex-critical.

%
%
Finally, let $H \in \mathcal{F}_5$ (see Fig.~\ref{druzina_F5}). 
If $e'=ab$, $a,b \in \{u,u_1,v,v_1\} \cup \{x_1, x_2, \ldots, x_l\}$, then $G$ contains a non-spanning subgraph isomorphic to a graph from $\{H_3, H_4\}$, or a non-spanning subgraph obtained by attaching a single leaf to two adjacent vertices of $C_n, n \geq 4$. By Theorem \ref{izrek_4kriticni_grafi} and Lemma \ref{lema_sosednji_vozlisci_stopnje_3_cikel_vsaj4barve}, $G$ is not  $4$-$\pch$-vertex-critical. 
If $e' \in \{uu_2,u_1u_3, vv_2, v_1v_3\}$, then $G$ contains a subgraph isomorphic to a graph from $\mathcal{F}_2 \cup \mathcal{F}_3$. By the previous consideration, $G$ is not  $4$-$\pch$-vertex-critical. 
Now, suppose that $e'=u_3x_i$ or $e'=v_3x_i$, where $i \in \{1,2, \ldots, l\}$. In this case, $G$ contains a non-spanning subgraph isomorphic to a graph from $\{H_4\} \cup \mathcal{C} \cup \mathcal{F}_5$. Then, Theorems \ref{thm:kriticni_CIKLI_rall_klavzar_izrek},  \ref{izrek_4kriticni_grafi} and \ref{izrek:druzine_so_4kriticne} imply that $G$ is not  $4$-$\pch$-vertex-critical.
If $e' \in E(G) \setminus \{ab, uu_2,u_1u_3, vv_2, v_1v_3, u_3x_i, v_3x_i; ~ 1 \leq i \leq l\ \land a,b \in \{u,u_1,v,v_1\} \cup \{x_1, x_2, \ldots, x_l\}\}$, then $G$ contains a subgraph isomorphic to $H_7$ (note that $G \not \cong H_7$), a non-spanning subgraph isomorphic to a graph from $\mathcal{C} \cup \mathcal{F}_1$ or a non-spanning subgraph obtained by attaching a single leaf to two vertices of $C_n, n \geq 4$, which are at distance $2k+1, k \geq 0$. Using the above consideration, Theorems \ref{thm:kriticni_CIKLI_rall_klavzar_izrek}, \ref{izrek:druzine_so_4kriticne} and Lemma \ref{lemma:cikel+2lista_na_lihi_razdalji_pch_vsaj_4}, we derive that $G$ is not a $4$-$\pch$-vertex-critical graph. This completes the proof. 
 \qed
\end{proof}

%
%
%
%
%

\begin{theorem}
A graph $G$ is $4$-$\pch$-vertex-critical if and only if $G \in \mathcal{F}_1 \cup \mathcal{F}_2 \cup \mathcal{F}_3 \cup \mathcal{F}_4 \cup \mathcal{F}_5 \cup \mathcal{C}_5 \cup \mathcal{C}_6 \cup \mathcal{C} \cup \{K_4, H_1, H_2, \ldots, H_{8}, H_9\}$. 
\label{glavni_izrek}
\end{theorem}

\begin{proof}

Let $\mathcal{X}=\mathcal{F}_1 \cup \mathcal{F}_2 \cup \mathcal{F}_3 \cup \mathcal{F}_4 \cup \mathcal{F}_5 \cup \mathcal{C}_5 \cup \mathcal{C}_6 \cup \mathcal{C} \cup \{K_4, H_1, H_2, \ldots, H_{8}, H_9\}$. If $G \in \mathcal{X}$, then Theorems  \ref{thm:kriticni_CIKLI_rall_klavzar_izrek}, \ref{izrek_4kriticni_grafi} and \ref{izrek:druzine_so_4kriticne} imply that $G$ is $4$-$\pch$-vertex-critical. 

To prove the converse implication, let $G$ be a $4$-$\pch$-vertex-critical graph. The assumption that $G \notin \mathcal{X}$ will lead us to a contradiction.
%
%
%
%
Note that $G$ is connected, and by Lemma \ref{lemma:podgrafi_4-kriticnost}, $G$ does not contain a subgraph isomorphic to a graph from $\mathcal{X}$.

In the sequel of this proof, we distinguish two cases.


\textbf{Case 1.} $G$ contains a subgraph isomorphic to $T$, which is shown in Fig.~\ref{fig:T}.

\textbf{Case 1.1.} $T$ is a spanning subgraph of $G$ ($V(T)=V(G)$). 

Let the vertices of $T$ be denoted as is shown in Fig.~\ref{fig:T}. 
Since $G$ is $4$-$\pch$-vertex-critical, Lemma \ref{drevoT=3_barve} implies $G \not \cong T$. Then, $G$ contains at least one of the edges from $\{ab, ac, ad, ae, bc, be, cd, de, dy', ey'\}$. 
Using the facts that $G \notin \mathcal{X}$ and it does not contain a subgraph isomorphic to a graph from $\mathcal{X}$, we infer that $bc, be, cd, de \notin E(G)$. 
Next, it is easy to observe that $\pch(G+f)=3$ for any $f \in \{ab, ac, ad, ae, dy', ey'\}$. This means that $G$ contains at least two edges from  $\{ab, ac, ad, ae, dy', ey'\}$. 
First, consider the case when exactly two of the mentioned edges belong to $E(G)$. If $\{ab, dy'\} \subset E(G)$, $\{ac, ey'\} \subset E(G)$, $\{ad, dy'\} \subset E(G)$, $\{ad, ey'\} \subset E(G)$, $\{ae, dy'\} \subset E(G)$ or $\{ae, ey'\} \subset E(G)$, then $\pch(G)=3$, a contradiction to $G$ being a $4$-$\pch$-vertex-critical graph. Otherwise, $G$ contains a  subgraph isomorphic to $H_1, H_2$ or to a graph from $\mathcal{C}$. This is a contradiction, since $G$ does not contain a subgraph isomorphic to a graph from $\mathcal{X}$.
Hence, $G$ contains at least three edges from $\{ab, ac, ad, ae, dy', ey'\}$. By reasoning from the case when $G$ contains exactly two edges from this set, we derive a contradiction (since $G$ does not contain a subgraph isomorphic to a graph from $\mathcal{X}$) for any triple of the chosen edges. 


\textbf{Case 1.2.} $T$ is not a spanning subgraph of $G$ ($V(T) \neq V(G)$). 

Let $y \in V(G)$ be a vertex belonging to a subgraph of $G$ isomorphic to $T$, and corresponding to the vertex $y' \in V(T)$. 
Further, let $N_i=\{u;~d(u,y)=i\}$ for any $i \geq 1$ and let $N_0=\{y\}$. 
In the sequel of this proof, we present some properties of $G$, which ensure that $G \in \mathcal{G}_3$ (and this will contradict to our assumption).  


First, consider the vertices from $\bigcup_{j\equiv 1,3\pmod{4}} N_j$ with degree at least $3$. Let $A=\{x;~\deg(x)\geq 3, x \in N_j, j\equiv 1,3\pmod{4}\}$ and let $x \in A$ ($x \in N_j$) be an arbitrary vertex. \\
We claim that $x$ has exactly one neighbor from $N_{j-1}$. 
Suppose to the contrary that $x$ is adjacent to distinct vertices $x_1, x_2 \in N_{j-1}$ (note that in this case $j \geq 3$). Then, there exist $x_1', x_2' \in N_{j-2}$ such that $x_1x_1', x_2x_2' \in E(G)$. If $x_1'=x_2'$, then $G$ contains a subgraph isomorphic to a graph from $\mathcal{F}_5 \cup \mathcal{C}$ (the vertices corresponding to the vertices of degree $3$ in a graph from $\mathcal{F}_5$ are $y$ and $x_1'$), a contradiction to our assumption. 
Thus, $x_1' \neq x_2'$. If $j \geq 5$, we again derive that $G$ contains a subgraph isomorphic to a graph from $\mathcal{F}_5$ (in this case, the vertices corresponding to the vertices of degree $3$ in a graph from $\mathcal{F}_5$ are $y$ and $x$), which contradicts to our assumption.
Otherwise, $x_1'y, x_2'y \in E(G)$, which implies that $G$ contains a subgraph isomorphic to a graph from $\mathcal{C}$, again a contradiction to our assumption.  
Therefore, $x$ has exactly one neighbor from $N_{j-1}$. \\
Now, suppose that there exist $x_1,x_2,x_3 \in V(G)$ such that $xx_1, xx_2, xx_3 \in E(G)$,  $x_1 \in N_{j-1}$, $x_2 \in N_{j}$ and $x_3 \in N_{j} \cup N_{j+1}$. We prove that all vertices from $N_j \cap N(x)$ are adjacent exactly to $x$ and $x_1$ (have degree $2$). 
Since $x_2 \in N_j$, it has a neighbor from $N_{j-1}$. If $x_2x_1  \notin E(G)$, then there exists $x_2' \in N_{j-1}$ such that $x_2x_2' \in E(G)$. But then, $G$ contains a subgraph isomorphic to a graph from $\mathcal{F}_5$ (in this case the vertices corresponding to the vertices of degree $3$ in a graph from $\mathcal{F}_5$ are $y$/$x_1$ and $x$) or a subgraph isomorphic to a graph from $\mathcal{C}$. In both cases, we have a contradiction to our assumption.
Consequently, every vertex from $N_j$ adjacent to $x$ is also adjacent to $x_1$. Moreover, each such vertex $x_2$ has degree $2$. Indeed, if $x_2x_3 \in E(G)$, then $G$ contains a subgraph isomorphic to a graph from $\{H_1, K_4\}$, a contradiction to $G$ being $4$-$\pch$-vertex-critical. If there exists $x_2'' \in V(G)$ such that $x_2x_2'' \in E(G)$ ($x_2'' \neq x_3$), then $G$ contains a subgraph isomorphic to $\{H_1, H_3\}$, again a contradiction to $G$ being $4$-$\pch$-vertex-critical. 
Hence, each vertex from $N_j \cap N(x)$ has degree $2$; it is adjacent to $x$ and $x_1$. \\
%
%
Next, we claim that each vertex from $N_{j+1} \cap N(x)$ is a leaf.  Let $x_1,x_2,x_3 \in V(G)$ such that $xx_1, xx_2, xx_3 \in E(G)$, $x_1 \in N_{j-1}$, $x_2 \in N_{j} \cup N_{j+1}$ and $x_3 \in N_{j+1}$. If $x_2x_3 \in E(G)$, then $G$ contains a subgraph isomorphic to a graph from $\mathcal{F}_2 \cup \mathcal{F}_3$, a contradiction to our assumption. 
Further, if there exists $x_3' \in V(G)$ such that $x_3x_3' \in E(G)$ ($x_3' \neq x_2$), then $G$ contains a subgraph isomorphic to a graph from $\{H_3, H_4, H_5, C_5\} \cup \mathcal{F}_5$, again a contradiction to our assumption. 
Hence, each vertex from $N_{j+1} \cap N(x)$ is a leaf. \\
In conclusion, each $ x \in \bigcup_{j\equiv 1,3\pmod{4}} N_j$ either has degree at most $2$ or it belongs to the set $A$ (has degree at least $3$). In the second case, $x \in N_j$ and its neighbors satisfy the following properties.\\
\textbf{Property 1}: Vertex $x$ has exactly one neighbor from $N_{j-1}$, say $x_1$. \\
\textbf{Property 2}: All vertices from $N_j \cap N(x)$ are adjacent exactly to $x$ and $x_1$ (have degree $2$). \\
\textbf{Property 3}: Each vertex from $N_{j+1} \cap N(x)$ is a leaf. \\ 
We observe that these properties imply that $N[x] \setminus \{x_1\}$ induces a $T$-add attached to $x_1$. Therefore, each vertex from $A$ belongs to a $T$-add.\\
%
%
%
Next, we prove that $A \subseteq \bigcup_{j\equiv 1\pmod{4}} N_j$ or $A \subseteq \bigcup_{j\equiv 3\pmod{4}} N_j$. Clearly, if $|A| \leq 1$, then the statement holds. Hence, consider the case when $|A| \geq 2$. Suppose to the contrary that there exist $x_1, x_2 \in A$ such that $x_1 \in N_{j_1}, j_1 \equiv 1\pmod{4}$ and $x_2 \in N_{j_2}, j_2 \equiv 3\pmod{4}$. Let $P_1$ be a shortest $x_1$-$y$-path, $P_2$ a shortest $x_2$-$y$-path and $w \in P_1 \cap P_2$ such that the distance between $y$ and $w$ is the largest possible. Note that $w \in N_j, j \equiv 0,2\pmod{4}$ and consequently, $w \neq x_1, x_2$. 
Namely, if $w \in N_j, j \equiv 1,3\pmod{4}$, then $w \in A$ (note that $\deg(w) \geq 3$) and thus, by Property 3, each vertex from $N(w) \cap N_{j+1}$ is a leaf. This is implies that $x_1,x_2$ are both adjacent to $w$ (or one of them is $w$) and both belong to $N_{j+1}$ (or one to $N_j$), a contradiction, since $x_1 \in N_{j_1}, j_1 \equiv 1\pmod{4}$ and $x_2 \in N_{j_2}, j_2 \equiv 3\pmod{4}$.
%
%
%
%
Therefore, without loss of generality we may assume that $w \in N_j, j \equiv 0\pmod{4}$. Then, $d(w, x_1) \equiv 1\pmod{4}$ and $d(w, x_2) \equiv 3\pmod{4}$. This implies that $G$ contains a subgraph isomorphic to a graph from $\mathcal{F}_4$ (the vertices corresponding to some of the vertices of degree $3$ in a graph from $\mathcal{F}_4$ are $x_1, x_2$ and $w$), a contradiction to our assumption. Thus, $A \subseteq \bigcup_{j\equiv 1\pmod{4}} N_j$ or $A \subseteq \bigcup_{j\equiv 3\pmod{4}} N_j$. Without loss of generality we may assume that $A \subseteq \bigcup_{j\equiv 3\pmod{4}} N_j$.  \\
%
%
%
%
%
We summarize our findings as follows. Since $A \subseteq \bigcup_{j\equiv 3\pmod{4}} N_j$, each vertex from $\bigcup_{j\equiv 1\pmod{4}} N_j$ has degree at most $2$. Next, it is possible that $G$ contains $T$-adds, which are attached to the vertices from $\bigcup_{j\equiv 2\pmod{4}} N_j$. In particular, it is possible that $G$ contains subgraphs isomorphic to $K_3$, which satisfy the following property: one vertex of each such triangle belongs to $N_j$ and the other two to $N_{j+1}$ for some $j \equiv 2\pmod{4}$ (we have proved that $j+1 \equiv 3\pmod{4}$). Now, we check whether $G$ contains a subgraph isomorphic to $K_3$ of a different type (i.e., triangle, which does not satisfy the property that one of its vertices belongs to $N_j$ and the other two to $N_{j+1}$ for some $j \equiv 2\pmod{4}$). 

Let $z,z_1,z_2 \in V(G)$ such that $zz_1, zz_2, z_1z_2 \in E(G)$. Clearly, if $a \in V(G)\cap N_j$, $b \in V(G) \cap N_i$ and $|j-i| \geq 2$, then $ab \notin E(G)$. Therefore, we can assume that $z,z_1,z_2 \in N_j \cup N_{j+1}$, $j \geq 0$. Moreover, let $z \in N_j$ and $z_1, z_2 \in N_j \cup N_{j+1}$, $j \geq 0$. Further, let $z' \in V(G)$ such that $zz' \in E(G)$ and $z' \in N_{j-1}$ (or $z' \in N_{j+1}$ if $j=0$).
If $j \equiv 1,3\pmod{4}$, then $z \in A$ and consequently, it satisfies the Properties 1, 2 and 3. This implies that $z_1$ (respectively, $z_2$) is a leaf, or it has degree $2$ and is adjacent to $z$ and $z'$. In both cases, we have a contradiction, since $z_1$, $z_2$ do not satisfy the written properties. Therefore, $j\equiv 0,2\pmod{4}$. \\
Now, we prove that $z_1, z_2 \in N_{j+1}$. Suppose to the contrary that $z_1 \in N_j$. This implies that $j \geq 2$ ($j\neq 0$, because $|N_0|=1$). 
Note that $z_1$ has a neighbor from $N_{j-1}$. If $z'z_1 \in E(G)$, then $G$ contains a subgraph isomorphic to $H_1$, a contradiction to $G$ being $4$-$\pch$-vertex-critical. Therefore, there exists $z_1' \in N_{j-1}$ such that $z_1z_1' \in E(G)$. Since $z', z_1' \in N_{j-1}$, there exist $z'', z_1'' \in N_{j-2}$ such that $z_1'z_1'', z'z'' \in E(G)$. 
If $z_1''=z''$, then $G$ contains a graph from $\mathcal{C}$ (namely $C_5$) as a non-spanning subgraph, which contradicts to $G$ being $4$-$\pch$-vertex-critical. Otherwise, $G$ contains a subgraph isomorphic to $H_4$, a contradiction to our assumption. Thus, $z_1, z_2 \in N_{j+1}$. \\
Therefore, each subgraph of $G$ isomorphic to $K_3$ has exactly one vertex from $N_j$ and the other two vertices from $N_{j+1}$ for some $j\equiv 0,2\pmod{4}$. \\
Let $B=\{z;~z \in N_j, j\equiv 0, 2\pmod{4}, z$ belongs to a subgraph of $G$ isomorphic to $K_3\}$. \\ 
We prove that $B \subseteq \bigcup_{j\equiv 0\pmod{4}} N_j$ or $B \subseteq \bigcup_{j\equiv 2\pmod{4}} N_j$.
Let $z, z' \in B$ and suppose that $z \in N_{j_1}, j_1 \equiv 0\pmod{4}$ and $z' \in N_{j_2}, j_2 \equiv 2\pmod{4}$. Let $P_1'$ be a shortest $z$-$y$-path, $P_2'$ a shortest $z'$-$y$-path and $w' \in P_1' \cap P_2'$ such that the distance between $y$ and $w'$ is the largest possible. 
Note that $w' \in N_j, j \equiv 0,2\pmod{4}$. Namely, if $w' \in N_j, j \equiv 1,3\pmod{4}$, then $G$ contains a subgraph isomorphic to a graph from $\mathcal{F}_5$ (the vertices corresponding to the vertices of degree $3$ in a graph from $\mathcal{F}_5$ are $y$ and $w'$) or a subgraph isomorphic to $H_5$. In both cases, we have a contradiction to our assumption. Thus, $w' \in N_j, j \equiv 0,2\pmod{4}$. But then, $G$ contains a subgraph from $\mathcal{F}_1$ and Lemma \ref{lemma:podgrafi_4-kriticnost} again implies a contradiction to $G$ being $4$-$\pch$-vertex-critical. Thus, $B \subseteq \bigcup_{j\equiv 0\pmod{4}} N_j$ or $B\subseteq \bigcup_{j\equiv 2\pmod{4}} N_j$. 
Further, we prove that $B \subseteq \bigcup_{j\equiv 2\pmod{4}} N_j$.
%
%
First, consider the case when $A \neq \emptyset$. Let $x \in A$ and let  $z \in B$. Then, $x \in N_{j_1}, j_1\equiv 3\pmod{4}$ and $z\in N_{j_2}, j_2\equiv 0,2\pmod{4}$. Suppose that $j_2\equiv 0\pmod{4}$.
Let $Q_1$ be a shortest $x$-$y$-path, $Q_2$ a shortest $z$-$y$-path and $v \in Q_1 \cap Q_2$ such that the distance between $y$ and $v$ is the largest possible. Analogously as for $w'$, also in this case we prove that $v\in N_j, j \equiv 0,2\pmod{4}$ and without loss of generality assume that $v \in N_j, j \equiv 0\pmod{4}$. This implies that $G$ contains a subgraph isomorphic to a graph from $\mathcal{F}_2$, a contradiction to our assumption.
Hence, $z \in N_{j_2}, j_2\equiv 2\pmod{4}$ and consequently, $B\subseteq \bigcup_{j\equiv 2\pmod{4}} N_j$.  
Note that, if $|A| = 0$, then without loss of generality we may assume that $B \subseteq \bigcup_{j\equiv 2\pmod{4}} N_j$.\\
In conclusion, each subgraph of $G$ isomorphic to $K_3$ satisfies the above written property: one of its vertices, denoted by $z$, belongs to $N_j$ and the other two, denoted by $x_1,x_2$, to $N_{j+1}$, where $j \equiv 2\pmod{4}$. Moreover, if $x_1$ (respectively, $x_2$) belongs to $A$ (i.e., has degree at least $3$), then it satisfies the Properties 1,2 and 3. Therefore, the vertices from $(N[x_1] \cup N[x_2]) \setminus \{z\}$ induce a $T$-add to a vertex $z$. 

Now, we make a partition of $V(G)$ into the sets $V_0', V_1', \ldots, V_7'$ as follows. \\
$V_0'=\{u \in N_j;~ j  \equiv 1\pmod{4}, ~$deg$(u)=1\}$, \\ 
$V_1'=\{u \in N_j;~ (j  \equiv 0\pmod{4}) \land (\exists u' \in N_{j-1}, uu' \in E(G), ~$deg$(u')=2)\} \cup \{y\}$,  \\
$V_2'=\{u \in N_j;~j  \equiv 1,3\pmod{4}, ~$deg$(u)=2,  u$ does not belong to a subgraph of $G$ isomorphic to $K_3\}$, \\
$V_3'=\{u \in N_j;~ j  \equiv 2\pmod{4}\}$ (note that $B \subseteq V_3'$), \\
$V_4'=\{u \in N_j;~ j  \equiv 3\pmod{4}, ~$deg$(u)=1\}$, \\
$V_5'=\{u \in N_j;~ j  \equiv 3\pmod{4},~ $deg$(u)=2, u$ belongs to a subgraph of $G$ isomorphic to $K_3\}$,\\
$V_6'=\{u \in N_j;~ j  \equiv 3\pmod{4},~ $deg$(u) \geq 3\}=A$, \\
$V_7'=\{u \in N_j;~ (j  \equiv 0\pmod{4})\land  ( \exists u' \in N_{j-1}, uu' \in E(G), ~$deg$(u')\geq 3)\}$. \\
First, we claim that each vertex from $V(G)$ belongs to exactly one set of $V_0', V_1', \ldots, V_7'$. Let $u \in V(G)$ be an arbitrary vertex. Suppose that $u \in N_j, j \equiv 0\pmod{4}$. If $u=y$, then $u \in V_1'$. Otherwise, $u$ has a neighbor $u' \in N_{j-1}$. Clearly, $u'$ is not a leaf. If deg$(u') \geq 3$, then $u' \in A$ and using Property 3, we infer that $u$ is a leaf. Consequently, there does not exist $u'' \in N_{j-1}, uu'' \in E(G)$ with deg$(u'')=2$, which implies that $u$ is only in $V_7'$ (clearly, $u' \neq y$). If deg$(u')=2$, then analogously we derive that there does not exist $u''' \in N_{j-1}, uu''' \in E(G)$ with deg$(u''') \geq 3$. This means that $u$ is only in $V_1'$. 
Therefore, $u$ belongs either to $V_1'$ or to $V_7'$ and clearly, $u \notin V_0' \cup V_2' \cup V_3' \cup V_4' \cup V_5' \cup V_6'$. 
Next, let $u \in N_j, j \equiv 1\pmod{4}$. If $\deg(u) =1$, then $u \in V_0'$. 
Further, suppose that deg$(u)=2$. Note that $u$ does not belong to a subgraph of $G$ isomorphic to $K_3$, since each subgraph of $G$ isomorphic to $K_3$ has one vertex in $B \subseteq \bigcup_{j\equiv 2\pmod{4}} N_j$ and the other two in $\bigcup_{j\equiv 3\pmod{4}} N_j$. This implies that $u \in V_2'$. 
Next, suppose that deg$(u) \geq 3$.  Then $u \in A$ and consequently, $u \in \bigcup_{j\equiv 3\pmod{4}} N_j$, a contradiction since $u \in N_j, j \equiv 1\pmod{4}$. Therefore, $u$ belongs to exactly one of the sets $V_0'$, $V_2'$, and $u \notin V_1' \cup V_3' \cup V_4' \cup V_5' \cup V_6' \cup V_7'$. 
Next, using the definitions of the sets $V_0', V_1', \ldots, V_7'$, we derive that each $u \in N_j, j \equiv 2\pmod{4}$ belongs only to the set $V_3'$. 
Finally, let $u \in N_j, j \equiv 3\pmod{4}$. If $u$ is a leaf, then $u \in V_4'$. If deg$(u) \geq 3$, then $u \in V_6'$. Further, let deg$(u)=2$. If $u$ belongs to a subgraph of $G$ isomorphic to $K_3$, then $u \in V_5'$, and otherwise, $u \in V_2'$. Therefore, $u$ belongs to exactly one of the sets $V_2'$, $V_4'$, $V_5'$, $V_6'$ and $u \notin V_0' \cup V_1' \cup V_3' \cup V_7'$. These observations imply that $(V_0', V_1', \ldots, V_7')$ is a partition of $V(G)$.

As we mentioned above, our goal is to prove that  $G \in \mathcal{G}_3$. Hence, we show that $G$ is formed by taking a bipartite multigraph with bipartition $(V_1',V_3')$, subdividing every edge exactly once, adding leaves to some vertices in $V_1'\cup V_3'$, and then performing a single $T$-add to some vertices in $V_3'$.

First, we prove that $V_1' \cup V_3'$ induces a bipartite multigraph in $G$, in which each edge is subdivided exactly once. Clearly, each vertex from $V_1' \cup V_3'$ belongs to exactly one of the sets $V_1'$, $V_3'$. Therefore, we only need to prove that $d(v_1, v_1')=4k_1$, $k_1 \geq 1$, and $d(v_3, v_3')=4k_3$, $k_3 \geq 1$, for any $v_1, v_1' \in V_1'$ and any $v_3, v_3' \in V_3'$. 
Let $v_1, v_1'$ be two distinct vertices from $V_1'$. Further, let $P$ be a shortest $v_1$-$v_1'$-path,  $Q_1$ a shortest $v_1$-$y$-path, $Q_1'$ a shortest $v_1'$-$y$-path and $w \in Q_1 \cap Q_1'$ such that the distance between $y$ and $w$ is the largest possible.
Recall that $v_1, v_1' \in \bigcup_{j\equiv 0\pmod{4}} N_j$. If $w\in N_j, j \equiv 1,3 \pmod{4}$, then $w \in A$, because $\deg(w) \geq 3$. Using Property 3, we derive that $wv_1, wv_1' \in E(G)$ ($v_1,v_1'$ are leaves) and consequently, $v_1,v_1' \in V_7'$. Since $(V_0, V_1', \ldots, V_7')$ is a partition of $V(G)$, $v_1,v_1' \notin V_1'$, a contradiction. Therefore, $w\in N_j, j \equiv 0,2\pmod{4}$.
Clearly, if $w \in P$, then $d(v_1,v_1') = 4k_1$ for some $k_1 \geq 1$. 
Next, consider the case when $w \notin P$. 
If $P \cap (Q_1 \cup Q_1') =\{v_1, v_1'\}$, then $d(v_1, v_1') = 4k_1$ for some $k_1 \geq 1$. 
Namely, otherwise $G$ contains a subgraph isomorphic to a graph from $\mathcal{C}$, a contradiction to our assumption. 
%
%
Now, let $P \cap (Q_1 \cup Q_1') \neq \{v_1,v_1'\}$.  Without loss of generality we may assume that $|P \cap Q_1| \geq 2$ and let $w_1 \in P \cap Q_1$, $w_1 \neq v_1$, such that the distance between $v_1$ and $w_1$ is the largest possible. Further, let $w_1' \in P \cap Q_1'$ ($w_1'$ can be $v_1'$, but $w_1 \neq w_1'$) such that the distance between $v_1'$ and $w_1'$ is the largest possible.
Note that $d(v_1,w_1) \neq 1$. Namely, if $v_1$ and $w_1$ are adjacent, then $v_1 \in V_7'$ (note that deg$(w_1) \geq 3$), a contradiction. Analogously, if $v_1' \neq w_1'$, then $d(v_1',w_1') \neq 1$.
Moreover, $w_1\in N_j, j \equiv 0,2\pmod{4}$. Indeed, $w_1\in N_j, j \equiv 1,3\pmod{4}$ implies that $w_1 \in A$ and consequently, $w_1v_1 \in E(G)$, which is not true.
Analogously we derive that $w_1'\in N_j, j \equiv 0,2\pmod{4}$. 
Next, we observe that the vertices from $Q_1$ between $w_1$ and $w$, the vertices from $Q_1'$ between $w_1'$ and $w$, and the vertices from $P$ between $w_1$ and $w_1'$ form a cycle $C_n$ in $G$, $n \geq 3$. 
Note that $w_1w \notin E(G)$ since $w_1\in N_{j_1}, j_1 \equiv 0,2\pmod{4}$ and $w\in N_{j_2}, j_2 \equiv 0,2\pmod{4}$, but $j_1 \neq j_2$ (clearly, $w \notin P$, $w_1  \in P$, hence $w_1 \neq w$). Thus, $C_n$ is not a triangle. Therefore, $n \geq 4$ and moreover, $n=4a, a \geq 1$, since $G$ does not contain a subgraph from $\mathcal{C}$. 
%
%
This implies that $d(w_1,w_1')=4k$, $k \geq 1$, if $w_1\in N_{j_1}, w_1'\in N_{j_2},j_1, j_2 \equiv 0\pmod{4}$ or $w_1\in N_{j_1}, w_1'\in N_{j_2},j_1, j_2 \equiv 2\pmod{4}$, and otherwise $d(w_1,w_1')=4l+2$, $l \geq 0$. 
In the first case, $d(v_1,w_1)+d(v_1'+w_1')=4k', k' \geq 1$. Thus, $d(v_1,v_1')=4k_1, k_1 \geq 1$. In the second case, $d(v_1,w_1)+d(v_1'+w_1')=4l'+2, l' \geq 0$, and hence, $d(v_1,v_1')=4l_1, l_1 \geq 1$.
Analogously, we prove that $d(v_3,v_3') = 4k_3$ for some $k_3 \geq 1$. Therefore, $V_1' \cup V_3'$ induces a subdivided bipartite multigraph in $G$. 

Now, we prove that $V_2'$ is the set of all vertices obtained by subdivision of bipartite multigraph with bipartition $(V_1', V_3')$. We have to show that each $v_2 \in V_2'$ has exactly two neighbors: one from $V_1'$ and the other from $V_3'$.
Let $v_2 \in V_2'$ be an arbitrary vertex. Then, $v_2 \in N_j, j \equiv 1,3\pmod{4}$. 
First, consider the case when $j \equiv 1\pmod{4}$. Clearly, $v_2v \in E(G)$ for some $v \in N_{j-1}$ ($j-1 \equiv 0\pmod{4}$). 
This implies that $v \in V_1' \cup V_7'$. It is easy to observe that each vertex from $V_7'$ has a neighbor from $V_6'=A$. By Property 3, each vertex from $V_7'$ is a leaf. This implies that $v \in V_1'$. Therefore, $N(v_2) \cap N_{j-1} \subseteq V_1'$. Since the vertices from $V_1'$ are pairwise at distance $4a$, $a \geq 1$, we derive that $|N(v_2) \cap N_{j-1}|=1$ ($v_2$ has exactly one neighbor from $N_{j-1}$, more precisely, from $V_1'$). Suppose that $v_2u \in E(G)$, $u \in N_j$. Then, $u \in V_0' \cup V_2'$.  
By Lemma \ref{lemma:vodoravna_povezava}, $v_2$ belongs to a subgraph of $G$ isomorphic to $K_3$, a contradiction to the definition of $V_2'$. Since deg$(v_2)=2$, $v_2$ has one neighbor from $N_{j+1}$, namely from $V_3'$ (note that $j+1 \equiv 2\pmod{4}$). In conclusion, $v_2$ has one neighbor from $V_1'$ and the other from $V_3'$.
Next, suppose that $j \equiv 3\pmod{4}$. Clearly, $v_2v \in E(G)$ for some $v \in N_{j-1}$. Since $j-1 \equiv 2\pmod{4})$, we infer that $N(v_2) \cap N_{j-1} \subseteq V_3'$. Again, since the vertices from $V_3'$ are pairwise at distance $4a$, $a \geq 1$, we derive that $v_2$ has exactly one neighbor from $N_{j-1}$. 
Suppose that $v_2u \in E(G)$, $u \in N_j$. 
Again, by Lemma \ref{lemma:vodoravna_povezava}, $v_2$ belongs to a subgraph of $G$ isomorphic to $K_3$, a contradiction to the definition of $V_2'$. 
Since deg$(v_2)=2$, $v_2z \in E(G)$ for some $z \in N_{j+1}$. Note that $z \in V_1' \cup V_7'$. The fact that deg$(v_2)=2$ implies that $z \in V_1'$. 
In conclusion, $v_2$ has one neighbor from $V_1'$ and the other from $V_3'$.

Further, let $v_1 \in V_1'$ ($v_1 \in N_j$, $j  \equiv 0\pmod{4}$) be an arbitrary vertex. 
We need to prove that $N(v_1)$ can contain only leaves and the vertices from $V_2'$. 
First, consider the case when $v_1 \neq y$. 
Since $v_1 \in N_j$, $j  \equiv 0\pmod{4}$, we have:  $N(v_1) \subseteq V_0' \cup V_1' \cup V_2' \cup V_4' \cup  V_5' \cup  V_6' \cup  V_7'$. 
Clearly, $v_1$ does not have a neighbor from $V_1'$ since $d(v_1,v_1')=4k_1, k_1 \geq 1$, for each $v_1' \in V_1'$, $v_1' \neq v_1$. 
If there exists $v_7 \in V_7'$ such that $v_1v_7 \in E(G)$, then deg$(v_7) \geq 2$ ($v_7$ has a neighbor $v_1$ and a neighbor from $N_{j-1}$), a contradiction to $v_7$ being a leaf (we have already argued that each vertex from $V_7'$ is a leaf). Thus, $N(v_1)$ does not contain the vertices from $V_7'$. 
Next, suppose that $v_1v_6 \in E(G)$ for some $v_6 \in V_6'$. By the definition of $V_1'$, $v_1$ is adjacent to $u' \in N_{j-1}$ such that deg$(u')=2$, which implies that $u' \notin V_6'$. Therefore, $v_1$ has at least two neighbors ($u'$ and $v_6$). On the other hand, since $v_6 \in A$, Property 3 implies that $v_1$ is a leaf, a contradiction. Hence, $N(v_1)$ does not contain the vertices from $V_6'$. 
Clearly, $v_1$ is not adjacent to a vertex from $V_4'$ since each vertex from $V_4'$ (from $N_{j-1}$) is a leaf and hence does not have a neighbor from $N_j$. 
Further, recall that each subgraph of $G$ isomorphic to $K_3$ has one vertex from $N_{j_1}, j_1  \equiv 2\pmod{4}$ and the other two from $N_{j_2}$, $j_2  \equiv 3\pmod{4}$. Next, each vertex from $V_5'$ belongs to a triangle and has degree $2$. These facts imply that $v_1$ does not have a neighbor from $V_5'$. Thus, $N(v_1) \subseteq V_0' \cup V_2'$. Obviously, $N(y)$ also contains only the vertices from $V_0' \cup V_2'$. Since each vertex from $V_0'$ is a leaf, our claim holds. 

It remains to prove that to each vertex $v_3 \in V_3'$ ($v_3 \in N_j$, $j  \equiv 2\pmod{4}$) only a single $T$-add, leaves and the vertices from $V_2'$ can be attached. 
Since $v_3 \in N_j$, $j  \equiv 2\pmod{4}$, we have: $N(v_3) \subseteq V_0' \cup V_2' \cup V_3' \cup V_4' \cup V_5' \cup V_6'$. 
%
%
Clearly, $v_3$ does not have a neighbor from $V_3'$ since $d(v_3,v_3')=4k_3, k_3 \geq 1$, for each $v_3' \in V_3'$, $v_3' \neq v_3$. 
Note that $v_3$ is not adjacent to a vertex from $V_0'$, since each vertex from $V_0'$ (from $N_{j-1}$) is a leaf and hence does not have a neighbor from $N_j$. 
Next, suppose that $v_5 \in V_5'$ is adjacent to $v_3$. Note that $v_5$ belongs to a triangle and has degree $2$. 
This implies that $v_3$ also belongs to this triangle (with vertices $v_5, v_3, b;~ b \in V_5' \cup V_6'$) and thus, $v_3 \in B$. As we have argued above, the vertices from $(N[v_5] \cup N[b]) \setminus \{v_3\}$ induce a $T$-add attached to $v_3$. Now, we prove that only one $T$-add can be attached to $v_3$.
Suppose that there exist $x_1, x_2 \in N_{j+1} \setminus \{v_5,b\}$ such that $\{v_3, x_1, x_2\}$ also induce a $K_3$. Then, $G$ contains a subgraph isomorphic to $H_2$, a contradiction to $G$ being $4$-$\pch$-vertex-critical. 
Next, suppose that  there exists $x_1 \in (N_{j+1} \cap N(v_3)) \setminus \{b\}$ such that $\deg(x_1) \geq 3$ ($x_1 \in V_6'$). Then, $G$ contains a subgraph isomorphic to a graph from $\{H_2\} \cup \mathcal{F}_2$ and Lemma \ref{lemma:podgrafi_4-kriticnost} contradicts that $G$ is a $4$-$\pch$-vertex-critical graph. 
Therefore, if $v_3$ is adjacent to a vertex from $V_5'$, only one $T$-add can be attached to $v_3$, and perhaps some leaves and the vertices from $V_2'$ (note that each vertex from $V_5' \cup V_6'$ belongs to a $T$-add). 
Thus, we can now assume that $N(v_3) \cap V_5' = \emptyset$. Suppose that $v_3v_6 \in E(G)$ for some $v_6 \in V_6'$. Note that $v_3$ does not belong to a subgraph of $G$ isomorphic to $K_3$, because $N(v_3) \cap V_5' = \emptyset$ and $G$ does not contains a subgraph isomorphic to $H_1$ or $H_3$. 
Therefore, since $v_6 \in A$, the Properties $1,2,3$ imply that all vertices from $N(v_6) \setminus \{v_3\}$ are leaves. This means that $N[v_6] \setminus \{v_3\}$ induces a $T$-add attached to $v_3$. Now, suppose that two $T$-adds can be attached to $v_3$. Since $N(v_3) \cap V_5' = \emptyset$, $v_3$ is adjacent to two vertices from $V_6'$. Since each vertex from $V_6'$ belongs to $A$ and $v_3$ does not belong to any triangle, the Properties 1,2,3 imply that $G$ contains a subgraph isomorphic to $H_9$, a contradiction to our assumption.
Therefore, also in the case when $N(v_3) \cap V_5' = \emptyset$ and $v_3$ has a neighbor from $V_6'$ we derive that only one $T$-add can be attached to $v_3$ (and perhaps some leaves and the vertices from $V_2'$).
If $N(v_3) \cap (V_5' \cup V_6') = \emptyset$, then only the vertices from $V_2' \cup V_4'$ can be adjacent to $v_3$. In this case, our claim clearly holds. 

Thus, $G \in \mathcal{G}_3$, a contradiction to $G$ being $4$-$\pch$-vertex-critical. 


\textbf{Case 2.} $G$ does not contain a subgraph isomorphic to $T$, which is shown in Fig.~\ref{fig:T}.

First, we consider the case when $G$ contains a subgraph isomorphic to $C_n$, $n \geq 3$. Lemma \ref{lemma:podgrafi_4-kriticnost} implies that $n=3$ or $n=4a, a \geq 1$. Suppose that $n=4a, a \geq 1$. Since $\pch(G)=4$, $G \not \cong C_n$. Thus, if $a \geq 2$, $G$ contains a subgraph isomorphic to $T$ or a subgraph isomorphic to a graph from $\mathcal{C}$, a contradiction to our assumptions. Hence, $a=1$. 
Let $V(C_4)=\{a,b,c,d\}$ and $\{ab,bc,cd,da\} \subseteq E(G)$. Note that $G$ is not isomorphic to $K_4$ and it does not contain a subgraph isomorphic to $K_4$. Hence, without loss of generality we may assume that $bd \notin E(G)$. 
Since $\pch(G)=4$, $G$ is not isomorphic to $H_1$ and $G$ does not contain a subgraph isomorphic to $H_1$, without loss of generality we may also assume that there exists a vertex from $V(G) \setminus \{b,c,d\}$ adjacent to $a$.
This implies that $\deg(b)=\deg(d)=2$, because $G$ does not contain a subgraph isomorphic to a graph from $\{H_5\} \cup \mathcal{C}$. 
If each vertex adjacent to $a$ is adjacent to $c$ and each vertex adjacent to $c$ is adjacent to $a$, then $\pch(G) \leq 3$ (recall that $G$ does not contain a subgraph isomorphic to $H_1, H_5$ or $K_4$), a contradiction to our assumption. 
Thus, there exists a vertex adjacent to $a$, which is not adjacent to $c$ (or a vertex adjacent to $c$ that is not adjacent to $a$). 
Next, note that if all vertices adjacent to $a$, which are not adjacent to $c$, are leaves, and also all vertices from $N(c)\setminus N[a]$ are leaves, then $\pch(G)=3$, again a contradiction. Therefore, there exist $a', a'' \in V(G)$, such that $aa',a'a'' \in E(G)$ and $a'c \notin E(G)$ (note that $a'' \neq b,c,d$). But then, $G$ contains a subgraph isomorphic to $T$, a contradiction. 
These findings imply that $n \neq 4$ (therefore, $G$ does not contain a subgraph isomorphic to $C_n$, $n \geq 4$), which means that $n=3$. 
Let $V(C_3)=\{a,b,c\}$ and let $E(C_3)=\{ab,bc,ca\}$. 
Since $\pch(G)=4$, at least one vertex from $\{a,b,c\}$ has degree at least $3$. 
Note that at least one vertex from $\{a,b,c\}$ has degree $2$. Namely, otherwise, $G$ contains a subgraph isomorphic to $H_3$ or to $C_4$, a contradiciton to our assumption. Therefore, without loss of generality we may assume that deg$(c)=2$.
Now, suppose that deg$(a)$, deg$(b) \geq 3$. This implies that there exist $a',b' \in V(G)$ such that $aa',bb' \in E(G)$. Note that $a' \neq b'$ ($a'b, b'a \notin E(G)$) since $G$ does not contain a subgraph isomorphic to $C_4$. 
Moreover, each vertex from $(N(a) \cup N(b)) \setminus \{a,b,c\}$ has exactly one neighbor in $N(a) \cup N(b) \cup \{a,b,c\}$, because $G$ does not contain a subgraph isomorphic to $H_2$ or $C_4$. 
Hence, since $\pch(G)=4$, there exists $a'' \in V(G) \setminus (N(a) \cup N(b) \cup \{a,b,c\}) $ (or $b'' \in V(G) \setminus (N(a) \cup N(b) \cup \{a,b,c\}$) such that $a'a'' \in E(G)$ (or $b'b'' \in E(G)$). Then, $G$ contains a subgraph isomorphic to $T$, a contradiction.
Therefore, deg$(a)\geq 3$ and deg$(b)=2$. 
Let $N_0=\{a\}$ and let $N_i=\{u;~d(u,a)=i\}$ for each $i=1,2, \ldots, k$. 
Note that $k \geq 2$. Namely, $\pch(G)=4$ and $G$ does not contain a subgraph isomorphic to $H_2$, which imply that there exists $a' \in V(G)$ such that $d(a,a')=2$. 
Further, the fact that $G$ does not contain a subgraph isomorphic to $T$ implies that all vertices from $N_1 \cup N_2 \cup \ldots \cup N_{k-2}$ have degree at most $2$ and in addition, deg$(a)=3$.
Now, let deg$(u_k) \geq 2$ for some $u_k \in N_k$. This implies that there exist $u_{k-1} \in N_{k-1}$ and $v \in (N_k \cup N_{k-1})$ such that $\{u_k, u_{k-1}, v\}$ induces a triangle. 
The fact that $G$ does not contain a subgraph isomorphic to $T$ implies that deg$(u_{k-1})=3$, deg$(v)=2$, deg$(u_k)=2$.
Since $G$ does not contain a subgraph isomorphic to a graph from $\mathcal{F}_1$, $k-2=4m+3$, $m \geq 0$. But then, there exists a $3$-packing coloring $c'$ of $G$. 
Namely, let $c'(c)=c'(v)=1$, $c'(b)=c'(u_k)=2$, $c'(a)=c'(u_{k-1})=3$. Further, if $m \geq 1$, then let the vertices from a subgraph of $G$ isomorphic to $P_{k-2}$ be colored one after another using the following pattern of colors: $1,2,1,3, \ldots, 1,2,1$. Otherwise, these vertices color with the following color pattern: $1,2,1$. Clearly, $c'$ is a $3$-packing coloring of $G$, a contradiction to $G$ being $4$-$\pch$-vertex-critical.
Hence, deg$(u_k)=1$ for every $u_k \in N_k$. Since $G$ does not contain a subgraph isomorphic to $T$, 
all vertices from $N(u_{k-1})$ except one, are leaves. 
%
If deg$(u_{k-1})=2$, then Lemma \ref{lema:trikotnik+pot=3barve} implies that $G \in \mathcal{G}_3$, a contradiction to $G$ being $4$-$\pch$-vertex-critical. Therefore, deg$(u_{k-1}) \geq 3$. 
Since $G$ does not contain a subgraph isomorphic to a graph from $\mathcal{F}_2$, we derive that $k-2 \neq 0$ and $k-2 \neq 4l+2$ for any $l \geq 0$. In this case, there exists a $3$-packing coloring $c''$ of $G$, defined as follows. Let $c''(c)=1, c''(b)=2$, $c''(a)=3$. 
If $k-2=4l+1$, $l \geq 0$, or $k-2=4l', l' \geq 1$, then let $c''(u_{k-1})=2$ and the leaves from $N(u_{k-1})$ receive color $1$. If $l\geq 1$, then the other vertices color one after another using the following pattern of colors $1,2,1,3, \ldots$. The same pattern is used for any $l' \geq 1$. If $l=0$, then color the unclored vertex by $1$. 
If $k-2=4l''+3$, $l'' \geq 0$, then let $c''(u_{k-1})=3$ and the leaves from $N(u_{k-1})$ receive color $1$. Further, if $l''\geq 1$, then color the other vertices one after another using the following pattern of colors $1,2,1,3, \ldots$. Otherwise, color the remaining three vertices one after another by $1,2,1$. 
Clearly, in each case, the obtained coloring is a $3$-packing coloring of $G$, a contradiction to $G$ being $4$-$\pch$-vertex-critical. This implies that $G$ does not contain a subgraph isomorphic to $K_3$. Then, $G$ is a tree.
Clearly, if every vertex has degree at most 2, then $G$ is a path and thus, it is not $4$-$\pch$-vertex-critical, a contradiction to our assumption. Hence, there exists $y \in V(G)$ such that deg$(y) \geq 3$. Let $N_0=\{y\}$ and let $N_i=\{u;~d(u,y)=i\}$ for each $i=1,2, \ldots, k$.
Recall that $\pch(G)=4$, $G$ is a tree and $G$ does not contain a subgraph isomorphic to $T$. Thus, $k \geq 3$. Moreover, all vertices from $N_1 \cup N_2 \cup \ldots N_{k-2}$ have degree at most $2$, all vertices from $N(y)$, except one, are leaves, and $|N_{i}|=1$ for any $i \in \{2,3, \ldots, k-1\}$.
Further, deg$(u_k)=1$ for every $u_k \in N_k$. Therefore, $G$ contains at most two vertices of degree at least $3$: one is $y$ and let the other be $y'$ ($y' \in N_{k-1}$).
Next, the fact that $G$ is a tree implies that all vertices from $N(y') \setminus N_{k-2}$ are leaves. Now, we prove that $G$ is $3$-packing colorable. Let $c: V(G) \longrightarrow \{1,2,3\}$. 
First, suppose that $k-2=4l$, $l \geq 1$, or $k-2=4l'+1, l' \geq 0$. In this case, let $c(y)=3$, $c(y')=2$ and all leaves from $N(y) \cup N(y')$ receive color $1$. The remaining vertices (of $P_{k-2}$) color one after another using the following pattern of colors: $1,2,1,3$ (note that if $l'=0$, then the vertex receives color $1$). Clearly, $c$ is a $3$-packing coloring of $G$, thus $\pch(G) \leq 3$, a contradiction to $G$ being $4$-$\pch$-vertex-critical. 
Next, let $k-2=4l''+2$, $l'' \geq 0$. Then, let $c(y)=c(y')=2$ and let all leaves from $N(y) \cup N(y')$ receive color $1$. The remaining vertices (of $P_{k-2}$) color one after another using the pattern of colors: $1,3,1,2$ (if $l''=0$, then use only colors $1,3$). The described coloring $c$ is a $3$-packing coloring of $G$, a contradiction to $G$ being $4$-$\pch$-vertex-critical.
Finally, suppose that $k-2=4l'''+3$, $l''' \geq 0$. In this case, let $c(y)=c(y')=3$ and all leaves from $N(y) \cup N(y')$ color with color $1$. Next, the remaining vertices (of $P_{k-2}$) color one after another using the pattern of colors: $1,2,1,3$ (if $l'''=0$, then use only colors $1,2,1$). Again, this is a $3$-packing coloring of $G$, a contradiction to $G$ being $4$-$\pch$-vertex-critical. This concludes the proof. 
\qed
\end{proof}


\section{$4$-$\pch$-critical graphs}
\label{kriticni_zadnje}

Recall that a graph $G$ is a $\chi_\rho$-critical graph, if for every proper subgraph $H$ of $G$, $\pch(H)<\pch(G)$. If $G$ is $\chi_\rho$-critical and $\chi_\rho(G)=k$, then we say that $G$ is $k$-$\chi_\rho$-critical. In this section, we characterize $4$-$\pch$-critical graphs. 

Bre\v sar and Ferme~\cite{bf-2019} proved that every $\pch$-critical graph is also $\chi_\rho$-vertex-critical. In addition, in trees these two types of critical graphs coincide. The mentioned authors also provided two partial characterizations of $4$-$\pch$-critical graphs, which are given below. 

\begin{proposition}~\cite{bf-2019} 
If $G$ is a graph containing a cycle $C_n$, where $n\ge 5$ is an integer not divisible by $4$, then $G$ is a $4$-$\chi_\rho$-critical graph if and only if $G$ is isomorphic to $C_n$.
\label{kriticni_BF_cikli}
\end{proposition}

Recall that $\mathcal{D}$ is the class of graphs that contain exactly one cycle and have an arbitrary number of leaves  attached to each of the vertices of the cycle.

\begin{theorem}~\cite{bf-2019}
A graph $G \in \mathcal{D}$ is a $4$-$\chi_\rho$-critical graph if and only if $G$ is one of the following graphs:
\begin{itemize}
\item $G \cong C_n$, $n \geq 5$, $n \not\equiv 0\pmod{4}$;
\item $G$ is the net graph;
\item $G$ is obtained by attaching a single leaf to two adjacent vertices of $C_4$.
\end{itemize}
\label{kriticni_BF_znani}
\end{theorem}

Let $\mathcal{F}_1'$ be the subfamily of graphs from $\mathcal{F}_1$ for which $l \notin\{0, 4k+2; ~k \geq 0\}$ (see Fig.~\ref{Druzina_F1}). Next, we denote by $\mathcal{F}_3'$ the subfamily of graphs from $\mathcal{F}_3$, which do not contain the edge $u_2u_3$ (see Fig.~\ref{Druzina_F3}). Further, $\mathcal{F}_4'$ consists of all graphs from $\mathcal{F}_4$, which do not contain the edges $v_1y_{l-1}$ and $v_1w_2$ (see Fig.~\ref{druzina_F4}). Finally, let $\mathcal{F}_5'$ be the subfamily of graphs from $\mathcal{F}_5$, which do not contain the edges $u_2u_3$ and $v_2v_3$ (see Fig.~\ref{druzina_F5}). Note that each graph from $\mathcal{F}_4' \cup \mathcal{F}_5'$ is a tree.

\begin{theorem}
\label{glavni_izrek_2}
A graph $G$ is $4$-$\pch$-critical if and only if $G \in \mathcal{F}_1' \cup \mathcal{F}_2 \cup \mathcal{F}_3' \cup \mathcal{F}_4' \cup \mathcal{F}_5' \cup \mathcal{C} \cup \{K_4, H_1, H_2, H_3, H_4, H_5, H_9\}$. 
\end{theorem}

\begin{proof}
The fact that each $\pch$-critical graph is also $\chi_\rho$-vertex-critical and Theorem \ref{glavni_izrek} imply that $4$-$\pch$-critical graphs can only be the graphs from $\mathcal{F}_1 \cup \mathcal{F}_2 \cup \mathcal{F}_3 \cup \mathcal{F}_4 \cup \mathcal{F}_5 \cup \mathcal{C}_5 \cup \mathcal{C}_6 \cup \mathcal{C} \cup \{K_4, H_1, H_2, \ldots, H_9\}$.

From Proposition \ref{kriticni_BF_cikli}, we derive that the graphs from $\mathcal{C}_5 \cup \mathcal{C}_6$ are not $4$-$\pch$-critical. By Theorem \ref{kriticni_BF_znani}, the graphs from $\mathcal{C} \cup \{H_3, H_5\}$ are $4$-$\pch$-critical. In addition, Theorem \ref{kriticni_BF_znani} also shows that $H_7$ is not a $4$-$\pch$-critical graph. Consequently, $H_8$ is not $4$-$\pch$-critical, since it contains a proper subgraph isomorphic to $H_7$. Further, since each tree is a $\pch$-critical graph if and only if it is $\pch$-vertex-critical, we know that $H_9$ is $4$-$\pch$-critical. Clearly, $K_4$ is also a $4$-$\pch$-critical graph. 
Therefore, it remains to consider the graphs from $\mathcal{F}_1 \cup \mathcal{F}_2 \cup \mathcal{F}_3 \cup \mathcal{F}_4 \cup \mathcal{F}_5 \cup \{H_1, H_2, H_4, H_6\}$. 
Let $G \in  \mathcal{F}_1 \cup \mathcal{F}_2 \cup \mathcal{F}_3 \cup \mathcal{F}_4 \cup \mathcal{F}_5 \cup \{H_1, H_2, H_4, H_6\}$. By Theorem \ref{glavni_izrek}, $\pch(G)=4$. Hence, we only need to check if $\pch(G-e) \leq 3$ for every edge $e$ of $G$.
Suppose that $e=xy$ and $x$ is a leaf. We observe that $G-e$ is isomorphic to the disjoint union of $G-x$ and $K_1$. Since $G$ is $4$-$\pch$-vertex-critical, $\pch(G-x) \leq 3$, which implies that $\pch(G-e) \leq 3$. Therefore, we only need to consider the edges of $G$, which connect the vertices of degree at least $2$. 

Let $G \cong H_1$ (see Fig.~\ref{Trikotnik_drugi}). If $e \in \{ab, ad, bd\}$, then $G-e \in \{X_2, Y_1\}$. By Lemma \ref{lema:c4+pot=3barve}, $\pch(G-e)=3$. In the case when $e \in \{ae, de\}$, $G-e$ is isomorphic to $H_3-d$. Since $H_3$ is $4$-$\pch$-vertex-critical, we have $\pch(G-e) \leq 3$. Thus, $H_1$ is $4$-$\pch$-critical.
Now, suppose that $G \cong H_2$ (see Fig.~\ref{Trikotnik_tretji}). If $e \in \{ab, be, bc, bd\}$, then $G-e \cong X_2$ and Lemma \ref{lema:c4+pot=3barve} implies that $\pch(G-e)=3$. Otherwise, color the leaves of $G-e$ with color $1$ and the remaining vertices with colors $1,2,3$ such that the vertex of degree $4$ receives color $2$. The described coloring is a $3$-packing coloring of $G-e$, which  implies that $H_2$ is $4$-$\pch$-critical.
Next, let $G \cong H_4$ (see Fig.~\ref{h4}). If $e=cd$, then $G-e$ is a path and hence, it is $3$-packing colorable. Further, if $e \in \{bc, de\}$, then $G-e$ is the disjoint union of $X_2$ and $K_2$. By Lemma \ref{lema:c4+pot=3barve}, $\pch(G-e)=3$. Now, let $e=cg$. Then, a $3$-packing coloring $c'$ of $G-e$ can be formed as follows: $c'(a)=c'(c)=c'(e)=c'(g)=1$, $c'(b)=c'(f)=2$ and $c'(d)=3$. This implies that $\pch(G-e) \leq 3$, if $e=cg$. Analogously we prove that $\pch(G-e) \leq 3$ holds for $e=dg$. Therefore, $H_4$ is $4$-$\pch$-critical.
Further, we claim that $H_6$ (see Fig.~\ref{h6}) is not a  $4$-$\pch$-critical graph. Note that $H_6-ad$ is isomorphic to $H_5$ and $\pch(H_5)=4$. Hence, $\pch(H_6-ad)=4$ and the claim is true.

Now, let $G \in \mathcal{F}_1$ (see Fig.~\ref{Druzina_F1}). If $e \in E(G) \setminus \{u_1u_2, v_1v_2\}$, then $G-e$ is isomorphic to $X_n$, $n \geq 1$, or to the disjoint union of two graphs from $\{K_3, X_n$;~ $n \geq 1\}$. 
Using Lemma \ref{lema:c4+pot=3barve} and the fact that $\pch(K_3)=3$, we derive that $\pch(G-e)=3$. Next, let $e \in \{u_1u_2, v_1v_2\}$. If $l=0$ or $l=4k+2$, $k \geq 0$, then $G-e$ is isomorphic to a graph from $\mathcal{F}_2$. Thus, $\pch(G-e)=4$ and $G$ is not $4$-$\pch$-critical.
If $l=4k, k \geq 1$, then $G-e$ is isomorphic to $H-u_2$, where $H \in \mathcal{F}_3$. By Theorem \ref{izrek:druzine_so_4kriticne}, $\pch(G-e) \leq 3$ and consequently, $G$ is $4$-$\pch$-critical. 
Finally, let $l=4k+1$, $k \geq 0$. In this case, we can form a $3$-packing coloring of $G-e$ as follows. Let the vertices of the triangle receive colors $1,2,3$ such that the vertex of degree $3$ receives a color $3$. Further, color the leaves with color $1$, and the remaining vertices (of $P_{4k+2}$) one after another with the following pattern of colors: $1,2,1,3$. This implies that $G$ is $4$-$\pch$-critical. In conclusion, each graph from $\mathcal{F}_1'$ is $4$-$\pch$-critical.

Further, let $G \in \mathcal{F}_2$ (see Fig.~\ref{Druzina_F2}). If $e \in E(G) \setminus \{u_1u_2\}$, then $G-e$ is isomorphic to $X_n$, to a subgraph of $X_n$ or to the disjoint union of such graphs. Hence, by Lemma \ref{lema:c4+pot=3barve}, $\pch(G-e) \leq 3$. If $e=u_1u_2$, then $G-e$ is isomorphic to $H-\{u_2, v_2\}$, where $H \in \mathcal{F}_5$. Using Theorem \ref{izrek:druzine_so_4kriticne}, we infer that $\pch(G-e) \leq 3$. Hence, every $G \in \mathcal{F}_2$ is $4$-$\pch$-critical.

Now, consider a graph $G \in \mathcal{F}_3$ (see Fig.~\ref{Druzina_F3}). Clearly, if $u_2u_3 \in E(G)$, then $G-u_2u_3 \in \mathcal{F}_3$ and hence, $G$ is not $4$-$\pch$-critical. Thus, we only need to consider the case when $G \in \mathcal{F}_3'$. If $e \in E(G) \setminus \{u_1u_2, v_1v_2\}$, then $G-e$ is isomorphic to (a subgraph of) $X_n$, to a subgraph of $Y_n$ or to the disjoint union of such graphs. Using Lemma \ref{lema:c4+pot=3barve}, we derive that $\pch(G-e) \leq 3$. If $e=v_1v_2$, then $G-e$ is isomorphic to $H-v_2$, where $H \in \mathcal{F}_5$. By Theorem \ref{izrek:druzine_so_4kriticne}, $\pch(G-e) \leq 3$. Note that $\pch(G-u_1u_2) \leq 3$, since $u_2$ is a leaf. Hence, every graph $G \in \mathcal{F}_3'$ is $4$-$\pch$-critical.

Next, let $G \in \mathcal{F}_4$ (see Fig.~\ref{druzina_F4}). First, suppose that $v_1y_{l-1} \in E(G)$. Then $G-v_1y_{l-1} \in \mathcal{F}_4$ and by Theorem \ref{izrek:druzine_so_4kriticne}, $\pch(G-v_1y_{l-1})=4$. Hence, $G$ is not $4$-$\pch$-critical.  Analogously, we prove that $G$ is not a $4$-$\pch$-critical graph, if  $v_1w_2 \in E(G)$. If $v_1y_{l-1}, v_1w_2 \notin E(G)$, then $G$ is a tree and consequently, it is $4$-$\pch$-critical. Therefore, each graph from $\mathcal{F}_4'$ is  $4$-$\pch$-critical.

Finally, let $G \in \mathcal{F}_5$ (see Fig.~\ref{druzina_F5}). If $G$ contains at least one of the edges from \{$u_2u_3, v_2v_3\}$, then $G-e\in \mathcal{F}_5$ for $e \in \{u_2u_3,v_2v_3\}$. By Theorem \ref{izrek:druzine_so_4kriticne}, $\pch(G-e)=4$ and thus, $G$ is not $4$-$\pch$-critical. Otherwise, $G$ is a tree and clearly, it is $4$-$\pch$-critical. Therefore, every graph from $\mathcal{F}_5'$ is $4$-$\pch$-critical. This concludes the proof.
\qed
\end{proof}

\section{Concluding remarks and open problems}
In this paper, we have characterized $4$-$\pch$-vertex-critical graphs and $4$-$\pch$-critical graphs. 

It is easy to prove that every finite graph $G$ with $\pch(G)=k$  contains a $\pch$-critical subgraph $H$ with $\pch(H)=k$. Therefore, every graph $G$ with $\pch(G)=4$ can be constructed by adding edges to a graph from $\mathcal{F}'_1 \cup \mathcal{F}_2 \cup \mathcal{F}'_3 \cup \mathcal{F}'_4 \cup \mathcal{F}'_5 \cup \mathcal{C}_5 \cup \mathcal{C}_6 \cup \mathcal{C} \cup \{K_4, H_1, H_2, \ldots, H_9\}$ (or it belongs to this set). In this way, the family of graphs with packing chromatic number $4$ can be characterized.

Note that $S$-packing colorings generalize the notion of  packing colorings. Recall that, if $S=(a_1,a_2,\ldots, a_k)$ is a non-decreasing sequence of positive integers, then an $S$-packing $k$-coloring of a graph $G$ is a partition of $V(G)$ into sets $X_1,X_2,\ldots, X_k$ such that for each pair of distinct vertices in the set $X_i$, the distance between them is larger than $a_i$. The $S$-packing chromatic number of $G$, $\chi_S(G)$, is the smallest $k$ such that $G$ has an $S$-packing $k$-coloring~\cite{goddard-2012}. Clearly, a $k$-packing coloring coincides with an $S$-packing $k$-coloring where $S =(1,2,3, \ldots, k)$. Holub, Jakovac and Klav\v zar~\cite{hjk-2020} studied $S$-packing chromatic vertex-critical graphs. These are graphs $G$ with the property that $\chi_S(G-u) < \chi_S(G)$ holds for every $u \in V(G)$. Among other results, the authors in~\cite{hjk-2020} have partially characterized $4$-$\chi_S$-critical graphs when $s_1 > 1$. It would be natural to study also the $4$-$\chi_S$-critical graphs for different sequences $S$.

\section{Acknowledgement}
The author acknowledges the financial support from the Slovenian Research Agency (the research project J1-9109).



\begin{thebibliography}{99}

\bibitem{balogh-2018}
  J.~Balogh, A.~Kostochka, X.~Liu,
  Packing chromatic number of cubic graphs,
  Discrete Math.\ 341 (2018) 474--483.

\bibitem{balogh-2019}
J.~Balogh, A.~Kostochka, X.~Liu,
Packing chromatic number of subdivisions
of cubic graphs,
Graphs Combin.\ 35 (2019) 513--537.

\bibitem{bozovic}
D.~Bo\v zovi\'c, I.~Peterin,
A note on the packing chromatic number of lexicographic products,
Discrete Appl.\ Math.\ 293 (2021) 34--37.

  
  \bibitem{bf-2018a} 
  B.~Bre\v sar, J.~Ferme,
 Packing coloring of Sierpi\'{n}ski-type graphs, 
 Aequationes Math.\ 92 (2018) 1091--1118.
 
 \bibitem{bf-2018b}
 B.~Bre\v sar, J.~Ferme,
 An infinite family of subcubic graphs with unbounded packing chromatic number,
 Discrete Math.\ 341 (2018) 2337--2342.
 
\bibitem{bf-2019}
 B.~Bre\v sar, J.~Ferme,
 Graphs that are critical for the packing chromatic number,
Discuss. Math. Graph Theory\ (in press); https://doi.org/10.7151/dmgt.2298.

\bibitem{bfk}
 B.~Bre\v sar, J.~Ferme, K.~Kamenic\' ka,
 $S$-packing colorings of distance graphs $G(\mathbb {Z},\{2, t\})$,
arXiv:2005.10491 [math.CO] (21 May 2020).

\bibitem{survey}
 B.~Bre\v sar, J.~Ferme, S.~Klav\v zar, D.~F.~Rall,
 A survey on packing colorings, 
 Discuss. Math. Graph Theory\ 40 (2020) 923--970.
 
\bibitem{bgt-2020}
 B.~Bre\v sar, N.~Gastineau, O.~Togni,
 Packing colorings of subcubic outerplanar graphs, Aequationes Math.\ 94 (2020) 945--967.

\bibitem{bkr-2007}
B.~Bre\v sar, S.~Klav\v zar, D.F.~Rall,
On the packing chromatic number of Cartesian products, hexagonal lattice, and trees,
Discrete Appl.\ Math.\ 155 (2007) 2303--2311.

\bibitem{bkr-2016}
  B.~Bre\v sar, S.~Klav\v zar, D.F.~Rall,
  Packing chromatic number of base-3 Sierpi\' nski graphs,
  Graphs Combin.\ 32 (2016) 1313--1327.
  
  \bibitem{deng}
  F.~Deng, Z.~Shao, A.~Vesel,
  On the packing coloring of base-3 Sierpi\'{n}ski and $H$ graphs,
Aequationes Math.\ (2020); https://doi.org/10.1007/s00010-020-00747-w.
  
\bibitem{ekstein-2014}
  J.~Ekstein, P.~Holub, O.~Togni,
  The packing coloring of distance graphs $D(k,t)$,
  Discrete Appl.\ Math.\ 167 (2014) 100--106.
  

\bibitem{ekstein-2010}
  J.~Ekstein, J.~Fiala, P.~Holub, B.~Lidick\'y,
  The packing chromatic number of the square lattice is at least $12$,
  arXiv:1003.2291 [cs.DM] (11 Mar 2010).
  
 \bibitem{fiala-2009}
  J.~Fiala, S.~Klav\v{z}ar, B.~Lidick{\'y},
 The packing chromatic number of infinite product graphs,
  European J.\ Combin.\ 30 (2009) 1101--1113.
  
  
  
  \bibitem{finbow-2010}
A.~Finbow, D.F.~Rall,
On the packing chromatic number of some lattices,
Discrete Appl.\ Math.\ 158 (2010) 1224--1228.
  
 \bibitem{ffgm}
  J.~Fres\'an-Figueroa, D.~Gonz\'alez-Moreno, M.~ Olsen, 
  On the packing chromatic number of Moore graphs, Discrete Appl. \ Math. \ 289 (2021) 185--193.
 
 
\bibitem{ght-2019}
   N.~Gastineau, P.~Holub, O.~Togni,
   On the packing chromatic number of subcubic outerplanar graphs,
   Discrete Appl.\ Math.\ 255 (2019) 209--221.
   
  
 \bibitem{goddard-2008}
  W.~Goddard, S.M.~Hedetniemi, S.T.~Hedetniemi, J.M.~Harris, D.F.~Rall,
  Broadcast chromatic numbers of graphs,
  Ars Combin.\ 86 (2008) 33--49.
 
 \bibitem{goddard-2012}
  W.~Goddard, H.~Xu,
  The $S$-packing chromatic number of a graph,
  Discuss.\ Math.\ Graph Theory {32} (2012) 795--806.
  
  
 \bibitem{hjk-2020}
 P.~Holub, M.~Jakovac, S.~Klav\v zar,  
 $S$-packing chromatic vertex-critical graphs,
 Discrete Appl. \ Math. \ 285 (2020) 119--127.
 

 \bibitem{kr-2019}
  S.~Klav\v zar, D.F.~Rall,
  Packing chromatic vertex-critical graphs,
  Discrete Math.\ Theor.\ Comput.\ Sci.\ 21 (3) (2019) \#8, 18 pp. 
  
  
  \bibitem{korze-2014}
  D.~Kor\v ze, A.~Vesel,
  On the packing chromatic number of square and hexagonal lattice,
  Ars Math.\ Contemp.\ 7 (2014) 13--22.
  
  \bibitem{lb-2017}
   D.~La\"{i}che, I.~Bouchemakh, \'{E}.~Sopena,
   Packing coloring of some undirected and oriented coronae graphs, 
   Discuss.\ Math.\ Graph Theory 37 (2017) 665--690.
   
  \bibitem{alf}
R.~Lemdani, M.~Abbas, J.~Ferme,
Packing chromatic number of some particular graphs, 
Filomat \ 34 (2020) 3275--3286.

\bibitem{martin-2017}
  B.~Martin, F.~Raimondi, T.~Chen, J.~Martin,
  The packing chromatic number of the infinite square lattice is between $13$ and $15$,
  Discrete Appl.\ Math.\ 225 (2017) 136--142.
  
  \bibitem{soukal-2010}
R.~Soukal, P.~Holub,
A note on packing chromatic number of the square lattice,
Electron.\ J.\ Combin.\ 17 (2010) N\#17, 7pp.

\bibitem{togni-2014}
  O.~Togni,
  On packing colorings of distance graphs,
  Discrete Appl.\ Math.\ 167 (2014) 280--289.
  
  \bibitem{torres-2015}
  P.~Torres, M.~Valencia-Pabon,
  The packing chromatic number of hypercubes,
  Discrete Appl.\ Math.\ 190--191 (2015) 127--140.
  
  \bibitem{vesel_korze}
A.~Vesel, D.~Kor\v ze,
Packing coloring of generalized Sierpi\'nski graphs,
Discrete Math.\ Theor.\ Comput.\ Sci.\ 21 (3) (2019) \#7, 18 pp.


\end{thebibliography}
\end{document}